\documentclass[11pt,reqno]{amsart}

\usepackage[T1]{fontenc}
\usepackage{lmodern}
\usepackage{microtype}
\usepackage{mathtools}
\usepackage{amssymb}
\usepackage{mathrsfs}
\usepackage{bm}
\usepackage{xcolor}

\usepackage{enumitem}
\usepackage{needspace}

\usepackage{hyperref}
\usepackage{bookmark}
\usepackage[nameinlink,capitalize,noabbrev]{cleveref}

\microtypesetup{expansion=false}
\setlist[itemize]{
  leftmargin=2.2em,
  itemsep=0.15em,
  topsep=0.25em
}
\numberwithin{equation}{section}
\allowdisplaybreaks[1]

\hypersetup{
  hypertexnames=false,
  hidelinks,
  pdfauthor={Wei-Chun Chen, Wei-Ting Kao, and Chun-Kai Lien},
  pdftitle={The Jacobi Spectrum and Global Second Variation of Pansu Spheres},
  pdfsubject={The Jacobi Spectrum and Global Second Variation of Pansu Spheres},
  pdfkeywords={Heisenberg group, Pansu sphere, horizontal perimeter, Jacobi operator,
    real spherical harmonics, Hopf frequencies, Morse index, characteristic points,  variation}
}

\newcommand{\PansuResultBookmark}[3]{%
  \begingroup
  \edef\PansuBookmarkDest{\getrefbykeydefault{#2}{anchor}{}}%
  \if\relax\detokenize\expandafter{\PansuBookmarkDest}\relax
  \else
    \bookmark[level=#1,dest={\getrefbykeydefault{#2}{anchor}{}}]{%
      #3 \ref*{#2}: \nameref*{#2}}%
  \fi
  \endgroup
}

\theoremstyle{plain}
\newtheorem{proposition}{Proposition}[section]
\newtheorem{lemma}[proposition]{Lemma}
\newtheorem{theorem}[proposition]{Theorem}
\newtheorem{corollary}[proposition]{Corollary}

\theoremstyle{remark}
\newtheorem{remark}[proposition]{Remark}

\theoremstyle{definition}
\newtheorem{definition}[proposition]{Definition}
\newtheorem*{AIuse}{AI Usage Statement}

\newcommand{\RR}{\mathbb R}
\newcommand{\CC}{\mathbb C}
\newcommand{\NN}{\mathbb N}
\newcommand{\Hn}{\mathbb H^n}

\newcommand{\cE}{\mathcal E}
\newcommand{\cK}{\mathcal K}
\newcommand{\cH}{\mathcal H}

\newcommand{\PSphere}{\mathcal S}
\newcommand{\Sphere}{S^{2n-1}}
\newcommand{\gSphere}{{g_{S^{2n-1}}}}
\newcommand{\LTheta}{{L_\Theta}}
\newcommand{\gTheta}{{g_\Theta}}
\newcommand{\gradSphere}{\nabla_b^{S^{2n-1}}}
\newcommand{\DeltaSphere}{\Delta_b^{S^{2n-1}}}
\newcommand{\TSphere}{T_{S^{2n-1}}}
\newcommand{\CRD}{\xi_{S^{2n-1}}}
\newcommand{\Shape}{\mathfrak S}
\newcommand{\II}{\mathrm{II}_{\xi}}
\newcommand{\PH}{P_H}

\newcommand{\dd}{\,\mathrm d}
\newcommand{\abs}[1]{\lvert#1\rvert}
\newcommand{\norm}[1]{\lVert#1\rVert}
\newcommand{\ip}[2]{\langle #1,#2\rangle}

\DeclareMathOperator{\spn}{span}
\DeclareMathOperator{\Spec}{Spec}

\DeclareMathOperator{\tr}{tr}
\DeclareMathOperator{\Id}{Id}

\let\Re\relax
\let\Im\relax
\DeclareMathOperator{\Re}{Re}
\DeclareMathOperator{\Im}{Im}

\title[Jacobi Spectrum and Global Second Variation]{The Jacobi Spectrum and Global Second Variation of Pansu Spheres}
\author{Wei-Chun Chen}
\address{Department of Mathematics, National Taiwan University, Taipei 10617, Taiwan}
\email{b14201028@ntu.edu.tw}

\author{Wei-Ting Kao}
\address{Institute of Mathematics, Academia Sinica, Taipei 106319, Taiwan}
\email{wtkao@as.edu.tw}

\author{Chun-Kai Lien}
\address{Department of Mathematics, National Taiwan University, Taipei 10617, Taiwan}
\email{d13221002@ntu.edu.tw}
\date{}

\newcommand{\PansuCMapPath}{}
\IfFileExists{\detokenize{$TEXMFOUTPUT}/\jobname-pansu-oms.cmap}
  {\renewcommand{\PansuCMapPath}{\detokenize{$TEXMFOUTPUT}/}}
  {\IfFileExists{\detokenize{$TEXMF_OUTPUT_DIRECTORY}/\jobname-pansu-oms.cmap}
    {\renewcommand{\PansuCMapPath}{\detokenize{$TEXMF_OUTPUT_DIRECTORY}/}}{}}
\AtBeginDocument{%
}

\begin{document}
\begin{abstract}
  {For $n\ge2$, we establish second differentiability of
  horizontal perimeter for compact $C^2$ boundaries in $\Hn$ and
  the continuous extension of their affine second variation to
  $C^1$ vector fields along the boundary. For the unit Pansu sphere,
  the Hessian agrees with the closed Jacobi form across both
  characteristic poles,
  including variations which move them. We compute the full
  Friedrichs spectrum of its
  Jacobi operator. The operator has compact resolvent, Morse index
  one and nullity $2n+1$; its kernel consists exactly of normalized
  translation speeds. Consequently, the sphere is stable under
  volume-preserving ambient variations.}
\end{abstract}
\maketitle

\section{Introduction}
\label{sec:introduction}

{The regular Jacobi formula for horizontal perimeter is
defined away from the characteristic set of a hypersurface in the
Heisenberg group. At a characteristic point the horizontal normal
vanishes, and the normalized normal speed of a smooth field can be
unbounded. This makes it necessary to determine when the regular
Jacobi pairing represents the global Hessian. We first establish
second differentiability for general $C^2$ boundaries, then
identify the Pansu-sphere Hessian and compute its Jacobi spectrum.}

The geometry of Pansu spheres is closely related to the umbilic
and horizontal-shape structure studied by Cheng, Chiu, Hwang,
and Yang~\cite{CCHY}. Hladky and Pauls derived a second-variation
formula for variations supported away from the characteristic
locus~\cite{HP}.
{The first arXiv version
of their work, posted in 2007, also contains an argument for
volume-preserving stability in
$\mathbb H^2$, using pole cutoffs, spherical harmonics and a
radial inequality~\cite{HP}.}
{In the three-dimensional setting, Hurtado and Rosales proved
volume-preserving stability for variations that are $C^3$ away from
the characteristic poles and satisfy their area-density
differentiation conditions, including a parameter-independent
integrable majorant for the second derivative of the area density.
Their class includes variations that move the
poles~\cite[Definition~5.1, Theorem~5.9, and Appendix~B]{HurtadoRosales}.}

{On any $C^2$ boundary, the geometry of the characteristic
set permits differentiation of horizontal perimeter when $n\ge2$.
In a vertical graph chart, perimeter is the integral of the norm
of a horizontal cofactor $G$. Its second derivative has at most a
$|G|^{-1}$ singularity, and the automatic bound
$\operatorname{rank}D_zG\ge n$ gives the required sublevel-set
estimate. This differentiation allows nonisolated characteristic
sets. Balogh proved that characteristic sets of $C^2$
hypersurfaces have Euclidean and Heisenberg Hausdorff dimensions at
most $n$~\cite{Balogh}; related rank estimates appear in
\cite{CHY,HP}.}

{The unit Pansu sphere has two characteristic poles.
Estimates at these poles put every $C^1$ surface-field speed in the
closed energy domain and identify its Hessian with the closed
Jacobi form.}

{For a round Euclidean sphere, dilation gives the unique
negative Jacobi mode and translations span the kernel. Pansu spheres
are rotational hypersurfaces with constant horizontal mean curvature
that arise in the Heisenberg isoperimetric problem. After identifying
their global Hessian, we compute the Jacobi spectrum to determine
whether the same index and kernel statements hold.}
Montefalcone solved the radial
horizontal-Laplacian eigenvalue problem in hypergeometric form
\cite{MontefalconeSpectrum}. {He proved radial stability
for zero-mean radial test functions with square-integrable horizontal
gradients and square-integrable quotients by the distance to the
vertical axis, subject to his horizontal integration-by-parts
condition~\cite[Remark~4.2 and Proposition~4.9]{MontefalconeStability}.
He also proved local stability for compactly supported
volume-preserving normal variations on sufficiently small
patches~\cite[Definition~3.15 and Theorem~4.14]{MontefalconeStability}.}
{Montefalcone's radial horizontal-Laplacian eigenvalue
problem differs from the Jacobi spectral problem here: the Jacobi
operator has a singular potential, and we fix its domain at the poles
by closing its form in $L^2(d\sigma_H)$.} We compare the
operators and their normalizations in
Remark~\ref{rem:montefalcone-comparison}. We assemble all angular sectors
to obtain the complete spectrum and its multiplicities. The angular
decomposition uses real spherical harmonics resolved by their Hopf
frequency; its eigenvalues and dimensions follow from the classical
spherical-harmonic calculations~{\cite{Folland,CowlingKlimaSikora}}.

\subsection{Main results}
\label{subsec:intro-spectrum}

{Throughout, $n\ge2$.}
{For a bounded domain $\Omega\subset\Hn$ with $C^2$
boundary $\Sigma$, let $N$ be the outward unit normal for the
left-invariant Riemannian metric $\gTheta$ and $N_b$ its horizontal
component. A surface field $W$ is a section of $T\Hn|_\Sigma$;
it need not be tangent to $\Sigma$. Set
$\Sigma^*=\Sigma\setminus\{N_b=0\}$. On $\Sigma^*$, its normalized
normal speed is
\[
  \varphi_W=\frac{\langle W,N\rangle}{|N_b|}.
\]
For $h\in\RR$, put $\mathcal F_h=\PH-hV$. For $C^{1,1}$ surface fields
$W_1,W_2$, choose compactly supported $C^{1,1}$ ambient extensions
$\widetilde W_1,\widetilde W_2$. In the fixed coordinates $(z,t)$,
set $\Phi_{r,s}=\Id+r\widetilde W_1+s\widetilde W_2$ and
$\Omega_{r,s}=\Phi_{r,s}(\Omega)$. Then
\[
  \mathcal B_h(W_1,W_2)
  =\left.\partial_r\partial_s\right|_{(0,0)}
    \bigl[\PH(\partial\Omega_{r,s})-hV(\Omega_{r,s})\bigr].
\]
The value is independent of the chosen extensions.}

\PansuResultBookmark{2}{thm:c2-global-hessian}{Theorem}
{\begin{theorem}[Second variation for $C^2$ hypersurfaces]
  \label{thm:c2-global-hessian}
  Let $n\ge2$, and let $\Omega\subset\Hn$ be a bounded open set
  with compact $C^2$ boundary $\Sigma$, oriented by its outward
  unit normal. Let $F_s\colon\Sigma\to\Hn$ be a family with
  $F_0=\Id$, and write $\partial\Omega_s=F_s\Sigma$.
  Assume one of the following:
  \begin{enumerate}[label=\textup{(\roman*)},leftmargin=2.2em]
    \item $F_s=\Phi_s|_\Sigma$, where $\Phi_s$ is a jointly $C^3$
      family of ambient diffeomorphisms;
    \item $F_s=\Phi_s^X|_\Sigma$, where $\Phi_s^X$ is the flow of
      a $C^2$ ambient vector field $X$;
    \item $F_s=\Id+\sum_{i=1}^m s_iW_i$ in the fixed coordinates
      $(z,t)$, where $W_i\in C^{1,1}(\Sigma;\RR^{2n+1})$.
  \end{enumerate}
  Then, near $s=0$, both functions
  \[
    s\longmapsto\PH(F_s\Sigma),\qquad
    s\longmapsto V(\Omega_s)
  \]
  are $C^2$. Their first and second derivatives are the integrals
  of the corresponding material density derivatives, which are
  integrable across the characteristic set.

  For every $h\in\RR$, the affine second variation $\mathcal B_h$
  of $\mathcal F_h=\PH-hV$ extends uniquely to a continuous
  bilinear form on $C^1$ surface fields, with
  \[
    |\mathcal B_h(U,W)|
    \le C_{\Sigma,h}\|U\|_{C^1}\|W\|_{C^1}.
  \]
  If $H_\Sigma=h$ on $\Sigma^*$, then $\Sigma$ is stationary
  for $\mathcal F_h$ and $\mathcal B_h$ is its Hessian.
\end{theorem}}

The continuous extension in the theorem is an extension of the affine
Hessian as a bilinear form on $C^1$ fields. Its assertion of twice
differentiability concerns the specified variation families.

Let $\PSphere=\partial\Omega\subset\Hn$ be the unit Pansu sphere,
whose horizontal mean curvature is $2n$.
{Let $\cE_n$ and $Q$ be its energy space and closed Jacobi
form, defined in Definition~\ref{def:energy}, and write
$Q[\varphi]=Q(\varphi,\varphi)$.}

\PansuResultBookmark{2}{cor:intro-hessian}{Corollary}
{\begin{corollary}[Global Hessian of the Pansu sphere]
  \label{cor:intro-hessian}
  On the unit Pansu sphere, every $C^1$ surface field $W$ has
  $\varphi_W\in\cE_n$, and for all such fields $W_1,W_2$,
  \begin{equation}
    \label{eq:rank-global-jacobi}
    \mathcal B_{2n}(W_1,W_2)
    =Q(\varphi_{W_1},\varphi_{W_2}).
  \end{equation}
  If $\Phi_\tau$ is a jointly $C^3$ family of ambient
  diffeomorphisms with $\Phi_0=\Id$, or the flow of a $C^2$ ambient
  field, and $\varphi$ is its initial normalized speed, then
  \begin{equation}
    \label{eq:main-hessian}
    \left.\frac{d^2}{d\tau^2}\right|_{\tau=0}
    \left[\PH(\Phi_\tau(\PSphere))-2nV(\Phi_\tau(\Omega))\right]
    =Q[\varphi].
  \end{equation}
\end{corollary}}

{The corollary includes variations which move the two
characteristic poles and imposes no vanishing condition on surface
fields there. Its proof in Section~\ref{sec:pansu-hessian} also gives
$\chi_\varepsilon\varphi_W\to\varphi_W$ in $\cE_n$ for the pole
cutoffs used there.}
{The general $C^2$ theorem is proved in
Section~\ref{sec:general-variation}; Section~\ref{sec:singular}
establishes the Pansu-sphere corollary.}

{For the spectral result, set $I=(-\pi/2,\pi/2)$ and
$\NN_0=\{0,1,2,\ldots\}$. Write $d\sigma_H$ for horizontal
perimeter measure and set
\[
  \mathcal X=L^2(\PSphere,d\sigma_H).
\]
Elements of Lebesgue and Sobolev spaces are identified almost
everywhere. In the regular coordinates
$(u,\zeta)\in I\times\Sphere$, the Jacobi expression and perimeter
measure are
\[
  L=-\Delta_b^{\PSphere}-(2n-2)\sec^2u,
  \qquad
  d\sigma_H=\cos^{2n}u\,du\,d\omega,
\]
where $\Delta_b^{\PSphere}$ is nonpositive. We use $L$ for the
Friedrichs operator of $Q$. All Hilbert spaces, Jacobi forms and
spectral eigenfunctions below are over $\RR$.}

{Let
\[
  \mathcal I=\{(d,m)\in\NN_0^2:0\le m\le d,\ d-m\in2\NN_0\}.
\]
For $(d,m)\in\mathcal I$ and $\Lambda\ge0$, put
\[
  \Lambda_{d,m}=d(d+2n-2)-m^2,
  \qquad
  \varrho(\Lambda)
  =\frac{1+\sqrt{1+4((n-1)(n-2)+\Lambda)}}2,
\]
and write $\varrho_{d,m}=\varrho(\Lambda_{d,m})$.
These parameters arise from the angular decomposition and
the radial P\"oschl--Teller reduction.}

\Needspace{9\baselineskip}
\PansuResultBookmark{2}{thm:main-spectrum}{Theorem}
{\begin{theorem}[Jacobi spectrum]
  \label{thm:main-spectrum}
  The Friedrichs Jacobi operator $L$ of the unit Pansu sphere
  $\PSphere\subset\Hn$ has compact resolvent and
  \begin{equation}
    \label{eq:main-spectrum}
    \Spec(L)
    =\left\{(\varrho_{d,m}+k)^2-n^2:
      (d,m)\in\mathcal I,\ k\in\NN_0\right\}.
  \end{equation}

  On the energy space, the Morse index is one and
  the nullity is $2n+1$.
  The kernel consists exactly of the normalized normal
  components of infinitesimal Heisenberg translations.
\end{theorem}}

{For a Pansu sphere of radius $R>0$, its ambient index is the
supremum of the dimensions of subspaces of $C^2$ ambient fields on
which the Hessian of $\PH-(2n/R)V$ is negative definite.}

{The global Hessian identity, spectral index, and dilation
pairing together give the following consequence.}

\PansuResultBookmark{2}{cor:intro-stability}{Corollary}
\begin{corollary}[Stability of Pansu spheres]
  \label{cor:intro-stability}
  For $n\ge2$, every Pansu sphere is stable under volume-preserving
  jointly $C^3$ ambient families and flows of $C^2$ ambient fields.
  Its ambient Hessian has index one. For every $C^1$ surface
  field $W$ satisfying $\delta V[W]=0$, this Hessian vanishes on
  $(W,W)$ if and only if $\varphi_W$ is the normalized normal
  component of an infinitesimal Heisenberg translation.
\end{corollary}

{Pole estimates identify $\mathcal B_{2n}$ on the Pansu sphere with
the closed Jacobi form $Q$. An endpoint Hardy estimate identifies
the domain of $Q$. Real Hopf-frequency harmonics
then reduce the operator to radial P\"oschl--Teller problems with
Gegenbauer eigenfunctions.}

\begin{AIuse}
The authors developed the geometric approach, introduced the
parametrization used for the spectral analysis, and organized and
wrote the manuscript. ChatGPT 5.6 Sol identified relevant work by
Hladky and Pauls \cite{HP} and Montefalcone
\cite{MontefalconeSpectrum}, and assisted with spectral calculations
using the authors' parametrization. It also helped work out the
details of the authors' approach to general second variation and
complete the proof of the Hessian identity across the characteristic
poles, for which the authors had developed part of the argument.
The authors independently re-derived and verified all AI-assisted
mathematical arguments and take full responsibility for the content of
the article.
\end{AIuse}
\section{General second variation and Jacobi pairings}
\label{sec:general-variation}

{For complex coordinate vectors, use the Hermitian product
$\ip zw=\sum_{j=1}^n z_j\overline{w_j}$, linear in the first argument.}
The Heisenberg group $\Hn=\CC^n\times\RR$ has product
{$(z,t)(w,s)=(z+w,t+s+\Im\ip zw)$}
and Haar measure $dV=dz\dd t$.  Writing $z_j=x_j+iy_j$, we use the fixed
left-invariant frame
\[
  X_j=\partial_{x_j}+y_j\partial_t,\qquad
  Y_j=\partial_{y_j}-x_j\partial_t,\qquad T=\partial_t,
\]
with $[X_j,Y_j]=-2T$.
{Set $\Theta=dt-\Im\ip z{dz}$} and $\xi=\ker\Theta$.
{Define $J$ by $JX_j=Y_j$ and $JY_j=-X_j$, define the Levi metric on
$\xi$ by $\LTheta=\tfrac12d\Theta(\cdot,J\cdot)$, and set
$\gTheta=\LTheta+\Theta^2$.
The frame $X_1,Y_1,\ldots,X_n,Y_n,T$ is orthonormal for
$\gTheta$.  {On the horizontal bundle $\xi$, let $\nabla^\xi$
be the connection determined by $\nabla^\xi_U X_j=\nabla^\xi_U Y_j=0$
for every vector field $U$. It preserves $\LTheta$ and $J$.}
For real tangent vectors, $\langle\cdot,\cdot\rangle$ denotes the
Riemannian metric $\gTheta$; on horizontal vectors identified with
$\CC^n$, it is the real part of the Hermitian product.}

Let $\Sigma\subset\Hn$ be an oriented two-sided $C^2$ hypersurface with
Riemannian unit normal $N$ and Riemannian surface measure $d\sigma_{\Sigma}$.
Denote the horizontal projection of $N$ by $N_b$.
On the regular set
$\Sigma^*=\Sigma\setminus \{N_b=0\}$, put
\[
  N_b=N-\ip NT T,\qquad
  e_{2n}=-\frac{N_b}{\abs{N_b}},\qquad
  \alpha=\frac{\ip NT}{\abs{N_b}},\qquad
  d\sigma_H=\abs{N_b}\,d\sigma_{\Sigma}.
\]
Set $e_n=-Je_{2n}$, so $e_{2n}=Je_n$.  Thus
\[
  \frac{N}{\abs{N_b}}=-e_{2n}+\alpha T,\qquad
  T+\alpha e_{2n}\in T\Sigma.
\]
For a boundary, $N$ is outward and $e_{2n}$ is inward.
The horizontal tangent bundle and its $J$-invariant part are
\[
  \xi\cap T\Sigma^*=\xi'\oplus\RR e_n,\qquad
  \xi'=(\xi\cap T\Sigma^*)\cap J(\xi\cap T\Sigma^*).
\]

{Complete $e_n,e_{2n}$ locally to an $\LTheta$-orthonormal frame
$e_1,\ldots,e_{2n}$ with the first $2n-1$ vectors tangent to $\Sigma$.
On $\xi\cap T\Sigma^*$ define $J'$ by $J'|_{\xi'}=J$ and $J'e_n=0$.}

For $X,Y\in\xi\cap T\Sigma^*$, use the horizontal second fundamental
form and symmetric horizontal shape operator:
\[
  {\II(X,Y)=-\ip{\nabla^\xi_Ye_{2n}}X,\qquad
  \Shape(X)=-\nabla^\xi_Xe_{2n}+\alpha J'X.}
\]
{These definitions follow the conventions of
\cite[(2.8) and Proposition~2.2]{CCHY}. The symmetric operator
$\Shape$ also appears in Ritor\'e~\cite[p.~52]{RitoreCalibration}.}

Define the horizontal mean curvature (also called $p$-mean curvature) by
\[
  H_{\Sigma}=\sum_{a=1}^{2n-1}\II(e_a,e_a)=\tr\Shape.
\]

For an orthonormal frame $e_1,\ldots,e_{2n-1}$ of $\xi\cap T\Sigma^*$,
the horizontal tangential gradient is
$\nabla_b^{\Sigma}\psi=\sum_{a=1}^{2n-1}(e_a\psi)e_a$. Define the
modified horizontal tangential sub-Laplacian
$\Delta_b^{\Sigma}$ by
\[
  \int_{\Sigma}\ip{\nabla_b^{\Sigma}\psi}{\nabla_b^{\Sigma}\phi}\dd\sigma_H
  =-\int_{\Sigma}(\Delta_b^{\Sigma}\psi)\phi\dd\sigma_H
\]
for real $\psi$ and real compactly supported tests $\phi$.
Thus $\Delta_b^{\Sigma}$ is nonpositive.
The horizontal tangential divergence with respect to $d\sigma_H$ is denoted by
$\operatorname{div}_{d\sigma_H}$, so
$\Delta_b^{\Sigma}=\operatorname{div}_{d\sigma_H}\nabla_b^{\Sigma}$.

Under $(z,t)\mapsto(z,-t/2)$, our horizontal frame agrees with Montefalcone's.
Our divergence and $\Delta_b^\Sigma$ correspond to $\mathcal D_{HS}$
and $\mathcal L_{HS}$ in \cite[Definition~1.16]{MontefalconeStability},
respectively; $\mathcal L_{HS}$ is the drift-modified operator associated
with $\Delta_{HS}$.
In $H_1$, \cite[(2.12)--(2.14)]{CHMY} gives
$\nabla_b^\Sigma=\nabla_b^t$ and $\Delta_b^\Sigma=\Delta_b^t+2\alpha e_1$.

In the regular second variation formula of Hladky and Pauls
\cite[Lemma~7.2 and Section~8.2]{HP}, the outward horizontal normal
$\nu$ equals $-e_{2n}$. Accordingly, define the Jacobi potential by
\[
  q_{\Sigma}\coloneqq \abs{\II}^2+4e_n\alpha+4\alpha^2.
\]
By the symmetric and skew-symmetric decomposition of the horizontal
second fundamental form~\cite[Propositions~2.1 and~2.2]{CCHY},
$|\II|^2=|\Shape|^2+2(n-1)\alpha^2$. Hence
\begin{equation}
  \label{eq:potential-symmetric-shape}
  q_{\Sigma}=\abs{\Shape}^2+4e_n\alpha+2(n+1)\alpha^2.
\end{equation}

{The rank-based sublevel estimates below justify the first
and second parameter derivatives of horizontal perimeter along the
stated variation families, including at characteristic points.}

{A surface field $W$ is a vector field along $\Sigma$,
that is, a section of $T\Hn|_\Sigma$; it need not be tangent to
$\Sigma$. Its $C^1$, $C^{1,1}$, or $C^2$ regularity is understood
componentwise in the fixed frame $X_j,Y_j,T$. Such a field admits
an ambient extension of the same regularity near $\Sigma$; when
$\Sigma$ is compact, the extension may be chosen compactly supported.
For global integrals on noncompact $\Sigma$, we take $W$ compactly
supported.}

For a surface field $W$, put
\[
  \varphi_W=\frac{\langle W,N\rangle}{|N_b|},\qquad
  \delta V[W]=\int_\Sigma\varphi_W\,d\sigma_H
\]
when $\Sigma$ bounds a domain and $N$ is outward.
{$\varphi_W$ is a scalar on $\Sigma^*$, and, with the
orientation induced by $N$, its associated density satisfies}
\[
  {\varphi_W\,d\sigma_H
  =\langle W,N\rangle\,d\sigma_{\Sigma}
  =\left.(\iota_W dV)\right|_{T\Sigma}\qquad\text{on }\Sigma^*.}
\]
{The last expression is defined across the characteristic set,
so the first-variation density extends there even though $\varphi_W$
itself is defined only on $\Sigma^*$.} Write
$\nabla_b=\nabla_b^\Sigma$, and let $\operatorname{div}_H$ be divergence
with respect to $d\sigma_H$ on the full tangent bundle of $\Sigma^*$.
Thus $\Delta_b=\operatorname{div}_H\nabla_b$ is nonpositive.
With the potential in \eqref{eq:potential-symmetric-shape}, set
\[
  L=-\Delta_b-q_\Sigma,\qquad
  Q_\Sigma(\varphi,\psi)=\int_\Sigma
  (\langle\nabla_b\varphi,\nabla_b\psi\rangle-q_\Sigma\varphi\psi)
  \,d\sigma_H
\]
{for speeds $\varphi,\psi$ whose supports are compact subsets
of $\Sigma^*$. Let
$\Omega\subset\Hn$ be a bounded domain with $C^2$ boundary
$\Sigma$, and put $\mathcal F_h=\PH-hV$. For
$C^{1,1}$ surface fields $U,W$, choose compactly supported
$C^{1,1}$ extensions $\widetilde U,\widetilde W$ to $\Hn$.
In the fixed coordinates $(z,t)$, set}
\[
  {\Phi_{r,s}=\Id+r\widetilde U+s\widetilde W,
    \qquad \Omega_{r,s}=\Phi_{r,s}(\Omega).}
\]
{For sufficiently small $r,s$, $\Phi_{r,s}$ is a
diffeomorphism. Define the affine second variation by}
\[
  {\mathcal B_h(U,W)
    =\left.\frac{\partial^2}{\partial r\partial s}\right|_{(0,0)}
      \bigl[\PH(\partial\Omega_{r,s})-hV(\Omega_{r,s})\bigr].}
\]
{This definition is independent of the chosen extensions.
Proposition~\ref{gen:rig:prop:second-variation} shows that
$\mathcal B_h$ extends continuously to $C^1$ surface fields.}

\subsection{Rank bounds and differentiation}
\label{gen:app:rank-proof}
{Near a characteristic point a $C^2$ hypersurface is a vertical
graph $t=f(z)$. Let $F(z,t)=t-f(z)$ and define its ambient horizontal
gradient by
\[
  \nabla_HF=\sum_{j=1}^n\bigl((X_jF)X_j+(Y_jF)Y_j\bigr).
\]
Put $G=-\nabla_HF$ and $g=|G|$. Explicitly,
$G=\sum_{j=1}^n\bigl((\partial_{x_j}f-y_j)X_j+
(\partial_{y_j}f+x_j)Y_j\bigr)$; in graph coordinates we regard $G$
as an $\RR^{2n}$-valued map in this frame. Put
$\sigma=\operatorname{sgn}\langle N,T\rangle$. The horizontal cofactor
for the chosen normal is $-\sigma G$, since
$N=\sigma(T-G)/\sqrt{1+g^2}$. Thus $d\sigma_H=g\,dz$ and
$|\alpha|=g^{-1}$.
For fixed left-invariant horizontal vectors $U,V$, since $TF=1$ and
$[U,V]=-2\langle JU,V\rangle T$, we have}
\[
  {\langle\nabla^\xi_U G,V\rangle
  -\langle\nabla^\xi_V G,U\rangle
  =-[U,V]F=2\langle JU,V\rangle.}
\]
{Here $D_zG$ is the $2n\times2n$ real Jacobian of this map,
with rows and columns ordered by $x_1,y_1,\ldots,x_n,y_n$.
The fixed frame is $\nabla^\xi$-parallel and $G$ is independent
of $t$ in this frame, so in graph coordinates this identity becomes}
\begin{equation}
  \label{gen:rig:eq:fixed-curl}
  {D_zG-(D_zG)^\top=2J.}
\end{equation}

{The rank argument in
\cite[proof of Theorem~D, (6.9)--(6.10)]{CHY} now follows from
\eqref{gen:rig:eq:fixed-curl}: for $v,w\in\ker D_zG$, one has
$\langle Jv,w\rangle=0$. Thus $\ker D_zG$ is isotropic for the
nondegenerate symplectic form $\langle J\cdot,\cdot\rangle$, so
$\dim\ker D_zG\le n$ and $\operatorname{rank}D_zG\ge n$.}

For a smooth ambient variation, let $f_s=F_s\circ\iota$ be a deformed
chart and let $G_s$ be its horizontal cofactor. If $G_0(x_0)=0$ and
$D_xG_0(x_0)$ is invertible, the implicit function theorem gives a
unique nearby zero $x(s)$ of $G_s$ with
\[
  x(0)=x_0,\qquad
  \left.\partial_{s_i}x(s)\right|_{s=0}
  =-\bigl(D_xG_0(x_0)\bigr)^{-1}
    \left.\partial_{s_i}G_s(x_0)\right|_{s=0}.
\]
The characteristic point of the deformed surface is $f_s(x(s))$;
$f_s(x_0)=F_s(f_0(x_0))$ is the image of the original point and is
characteristic only when $G_s(x_0)=0$. At the Pansu poles,
\eqref{eq:exact-pole-cofactors} gives
$G_\pm(z)=(r/\sqrt{1-r^2}\,\Id\mp J)z$, where $r=|z|$.
Consequently $D_zG_\pm(0)=\mp J$ is invertible since $J^2=-\Id$;
the formula above applies at both poles.

{To analyze the second parameter derivative of the
horizontal cofactor norm $|G|$, first record the derivatives of the
Euclidean norm.} For $x\ne0$, set $\widehat x=x/|x|$ and
$\Pi_xv=v-\langle\widehat x,v\rangle\widehat x$. Differentiating the
normalization in a direction $w$ gives
\[
  D\widehat x[w]=\frac{w}{|x|}-\frac{x\langle x,w\rangle}{|x|^3}
               =\frac{\Pi_xw}{|x|}.
\]
Consequently,
\[
  \begin{aligned}
    D^2|x|[v,w]
    &=\langle D\widehat x[w],v\rangle
      =\frac{\langle v,w\rangle}{|x|}
       -\frac{\langle x,v\rangle\langle x,w\rangle}{|x|^3}\\
    &=\frac{\langle\Pi_xv,\Pi_xw\rangle}{|x|}.
  \end{aligned}
\]
Thus the first and second norm derivatives are
\begin{equation}
  \label{gen:eq:norm-first-second-derivatives}
  D\abs x[v]=\ip{\widehat x}v,
  \qquad
  D^2\abs x[v,w]=\frac{\ip{\Pi_xv}{\Pi_xw}}{\abs x}.
\end{equation}
{For a parameter-dependent map $G$, second parameter
derivatives of $|G|$ have at most a $|G|^{-1}$ singularity.
Lemma~\ref{gen:lem:rank-differentiation} shows that this factor is
locally uniformly integrable when $\operatorname{rank}D_zG\ge2$ at
the zeros.}

Let $U\Subset\RR^d$ be a bounded domain, $d\ge2$, and let
$\mathcal P\subset\RR^m$ be a neighborhood of $0$.
Suppose $G(s,z)\in\RR^d$ and $\beta(s,z)\in\RR$ are defined on a
neighborhood of $\mathcal P\times\overline U$, with
$D_s^jG,D_s^j\beta$ ($j=0,1,2$) and $D_zG$ jointly continuous.
For a fixed weight $w\in L^\infty(U)$, set
$I(s)=\int_U w(z)\beta(s,z)|G(s,z)|\dd z$.

\begin{samepage}
  \begin{lemma}
    \label{gen:lem:rank-differentiation}
    Assume $\operatorname{rank}D_zG(0,z)\ge r$ at every zero
    in the closure $\overline U$ for $2\le r\le d$. After shrinking
    $\mathcal P$, the zero set of
    $G(s,\cdot)$ is null for every $s$, and, uniformly in $s$,
    \begin{equation}
      \label{gen:eq:parameter-rank-caps}
      \int_{U\cap\{|G(s,\cdot)|<\rho\}}|G(s,z)|^{-\gamma}\dd z
      \le C_\gamma\rho^{r-\gamma},\qquad 0\le\gamma<r,
    \end{equation}
    for all sufficiently small $\rho>0$. Moreover,
    \[
      I\in C^2(\mathcal P),\qquad
      D^jI(s)=\int_U w(z)D_s^j(\beta|G|)(s,z)\dd z\quad(j=1,2).
    \]
    The derivative densities are defined almost everywhere and belong to $L^1(U)$.
  \end{lemma}
\end{samepage}
\begin{proof}
  Fix every zero $p$, relabel coordinates so that the rank minor
  $(\partial_{z_j}G_i(0,p))_{1\le i,j\le r}$ is invertible.
  The map
  \[
    \Psi(s,z)=(s,\chi_s(z)),\qquad
    \chi_s(z)=(G_1(s,z),\ldots,G_r(s,z),z_{r+1},\ldots,z_d)
  \]
  is $C^1$, and the determinant of $D\Psi(0,p)$ is the chosen nonzero minor.
  The inverse function theorem gives a local diffeomorphism.
  On a smaller fixed product $\mathcal P_p\times B_p$ with compact closure,
  each $\chi_s$ is injective with uniformly bounded inverse Jacobian.

  Write $(a,b)=\chi_s(z)$. The $b$ coordinates lie in a fixed bounded box
  $K\subset\RR^{d-r}$, and $|G(s,z)|\ge|a|$. Change of variables gives
  \[
    \begin{aligned}
      \int_{U\cap B_p\cap\{|G(s,\cdot)|<\rho\}}
        |G(s,z)|^{-\gamma}\,dz
      &\le C\int_K\int_{|a|<\rho}|a|^{-\gamma}\,da\,db\\
      &=C|K|\,|S^{r-1}|\int_0^\rho t^{r-1-\gamma}\,dt\\
      &=\frac{C|K|\,|S^{r-1}|}{r-\gamma}\rho^{r-\gamma},
        \qquad 0\le\gamma<r.
    \end{aligned}
  \]
  A finite subcover of the zero set, together with the uniform positivity of \(|G|\) on the compact complement, yields \eqref{gen:eq:parameter-rank-caps}. The cases $\gamma=0,1$ give null zero sets and integrability of $|G|^{-1}$.

  For $\varepsilon>0$, set $N_\varepsilon(G)=(|G|^2+\varepsilon^2)^{1/2}$
  and $I_\varepsilon=\int_Uw\beta N_\varepsilon(G)\,dz$.
  These integrals are $C^2$, and
  \[
    DN_\varepsilon(G)[v]=\frac{G\cdot v}{N_\varepsilon(G)},\qquad
    D^2N_\varepsilon(G)[v,w]
    =\frac{v\cdot w}{N_\varepsilon(G)}
    -\frac{(G\cdot v)(G\cdot w)}{N_\varepsilon(G)^3}.
  \]
  Write $G_i=\partial_{s_i}G$, $G_{ij}=\partial_{s_i}\partial_{s_j}G$
  and use the same notations for $\beta$. 
  The full product rule is
  \[
    \begin{aligned}
      \partial_{s_i}(\beta N_\varepsilon)
      &=\beta_iN_\varepsilon+\beta\frac{G\cdot G_i}{N_\varepsilon},\\
      \partial_{s_i}\partial_{s_j}(\beta N_\varepsilon)
      &=\beta_{ij}N_\varepsilon+\beta_i\frac{G\cdot G_j}{N_\varepsilon}
        +\beta_j\frac{G\cdot G_i}{N_\varepsilon}\\
      &\quad+\beta\left(
         \frac{G_i\cdot G_j+G\cdot G_{ij}}{N_\varepsilon}
        -\frac{(G\cdot G_i)(G\cdot G_j)}{N_\varepsilon^3}\right).
    \end{aligned}
  \]
  Since $N_\varepsilon\ge|G|$ and $|G|/N_\varepsilon\le1$,
  \[
    |DN_\varepsilon(G)[v]|\le|v|,\qquad
    |D^2N_\varepsilon(G)[v,w]|
       \le\frac{|v||w|}{N_\varepsilon}+\frac{|G|^2|v||w|}{N_\varepsilon^3}
       \le\frac{2|v||w|}{|G|}.
  \]
  Boundedness of the derivatives on compact sets now gives
  $|D_s^j(\beta N_\varepsilon(G))|\le C(1+|G|^{-1})$
  for $j\le2$ and $0<\varepsilon\le1$, away from the null zero set.
  On $\{|G|\ge\rho\}$, the density formulas are smooth
  in $\varepsilon^2$, with derivatives bounded uniformly in $s$.
  Hence
  \[
    |D_s^j(\beta N_\varepsilon(G))-D_s^j(\beta|G|)|
    \le C_\rho\varepsilon^2,\qquad j\le2.
  \]
  On $\{|G|<\rho\}$, the integrals of these densities and their
  limits are bounded by $C(\rho^r+\rho^{r-1})$.
  More explicitly, for any derivative $D_s^j$ of order $j\le2$,
  splitting the difference over $|G|\ge\rho$ and $|G|<\rho$ gives
  \[
    \sup_s\left|D_s^j I_\varepsilon(s)
        -\int_UwD_s^j(\beta|G|)\,dz\right|
    \le C_\rho\varepsilon^2+C(\rho^r+\rho^{r-1}),
  \]
  where the supremum is over a fixed compact set and,
  for each fixed $\rho>0$, $C_\rho$ is independent of $s$ and
  $\varepsilon$. First let
  $\varepsilon\downarrow0$, then $\rho\downarrow0$.
  Thus $I_\varepsilon$ and its first two derivatives converge locally uniformly, and hence we have $I\in C^2$.
\end{proof}

\begin{samepage}
  \begin{corollary}
    \label{gen:rig:lem:sublevel}
    {At a characteristic point $p$, $G=-\nabla_HF$ satisfies
    $r_p:=\operatorname{rank}D_zG(p)\ge n$.}
    {For every vertical graph domain
    $U\subset\RR^{2n}$ and $K\Subset U$,}
    \begin{equation}
      \label{gen:rig:eq:weighted-sublevel}
      \int_{K\cap\{g<\rho\}}g^{-\gamma}\,dz
      \le C_\gamma\rho^{n-\gamma}\qquad(0\le\gamma<n)
    \end{equation}
    for sufficiently small $\rho>0$. {In particular, the
    volume of $K\cap\{g<\rho\}$ is $O(\rho^n)$ and, for $n\ge2$,
    its $g^{-1}$ integral is $O(\rho^{n-1})$.}
    The characteristic set is null in graph coordinates,
    and $\Sigma^*$ is dense in $\Sigma$.
  \end{corollary}
\end{samepage}
\begin{proof}
  {Apply the rank-coordinate construction of
  Lemma~\ref{gen:lem:rank-differentiation} at each characteristic point,
  with $a\in\RR^{r_p}$ and $r_p\ge n$. It gives $g\ge|a|$.
  Finitely many such neighborhoods cover the zeros
  in $K$; $g$ is bounded below on the remaining
  compact set. Equation~\eqref{gen:eq:parameter-rank-caps} at $s=0$
  gives \eqref{gen:rig:eq:weighted-sublevel}, since
  $\rho^{r_p-\gamma}\le\rho^{n-\gamma}$ for $0<\rho<1$.
  Taking $\gamma=0$, and $\gamma=1$ when $n\ge2$, gives the two
  stated bounds. Since
  $\{G=0\}\subset\{a=0\}$ on each rank-coordinate neighborhood,
  the characteristic set is null and $\Sigma^*$ is dense.}
\end{proof}

\begin{remark}
  {The sublevel estimate \eqref{gen:rig:eq:weighted-sublevel}
  uses only $\operatorname{rank}D_zG(p)\ge n$, which holds at every
  characteristic point. It applies whether $D_zG(p)$ is invertible or
  singular; no assumption that the characteristic point is isolated
  is needed.}
\end{remark}

The automatic rank bound $\operatorname{rank}D_zG\ge n$ controls the
$g^{-1}$ term in the perimeter Hessian for $n\ge2$. The Jacobi-energy
estimate below involves $g^{-3}$; its rank-coordinate integral
requires $r_p>3$.

{For the next proposition, fix a compact vertical graph patch
with the characteristic point in its interior. Energy integrals are
over the regular part of each cap, and $\|W\|_{C^1}$ is computed on
this fixed surface patch. Constants may depend on the patch and its
$C^2$ graph, but not on $W$ or the cap radius.}

{\begin{proposition}[Local energy at a characteristic point]
  \label{gen:prop:characteristic-energy}
  Let $\Sigma$ be a $C^2$ hypersurface, written near a characteristic
  point $p=(z_p,f(z_p))$ as a vertical graph $t=f(z)$. Put
  $G=-\nabla_H(t-f)$, $g=|G|$, and
  $r_p=\operatorname{rank}D_zG(z_p)$. If $r_p\ge4$, then, for
  $0\le\gamma<r_p$ and all sufficiently small $\varepsilon>0$,
  \begin{equation}
    \label{gen:eq:characteristic-rank-cap}
    \int_{|z-z_p|<\varepsilon}g^{-\gamma}\,dz
    \le C_\gamma\varepsilon^{2n-\gamma}.
  \end{equation}
  Every $C^1$ surface field $W$ on this patch also satisfies
  \begin{equation}
    \label{gen:eq:characteristic-energy-cap}
    \int_{|z-z_p|<\varepsilon}
    \bigl(|\nabla_b\varphi_W|^2+(1+|q_\Sigma|)\varphi_W^2\bigr)
    \,d\sigma_H
    \le C\varepsilon^{2n-3}\|W\|_{C^1}^2.
  \end{equation}
  If $D_zG(z_p)$ is invertible, then $p$ is isolated and, with
  $r=|z-z_p|$,
  \begin{equation}
    \label{gen:eq:characteristic-energy-pointwise}
    |\varphi_W|\le Cr^{-1}\|W\|_{C^1},\qquad
    |\nabla_b\varphi_W|\le Cr^{-2}\|W\|_{C^1}.
  \end{equation}
\end{proposition}
\begin{proof}
  In the graph coordinates write
  $A_W=W^t-D_zf\cdot W^z$ and
  $\sigma=\operatorname{sgn}\langle N,T\rangle$. Then
  $\varphi_W=\sigma A_W/g$, $d\sigma_H=g\,dz$, and
  $|A_W|+|D_zA_W|\le C\|W\|_{C^1}$.
  Since $G$ is $C^1$, its derivative is bounded, and thus
  $|D_zg|\le C$ where $g>0$. Horizontal tangential unit vectors
  have Euclidean $z$-components of norm at most one, so
  \[
    |\varphi_W|\le Cg^{-1}\|W\|_{C^1},\qquad
    |\nabla_b\varphi_W|\le Cg^{-2}\|W\|_{C^1}
  \]
  after shrinking the patch so that $g\le1$.
  Moreover $e_{2n}$ is, up to a locally constant sign, $G/g$ in the
  fixed horizontal frame, while $|\alpha|=g^{-1}$. The definitions of
  $\II$ and $q_\Sigma$ therefore give
  $|\II|\le Cg^{-1}$, $|e_n\alpha|\le Cg^{-2}$, and
  $|q_\Sigma|\le Cg^{-2}$. Hence the integral in
  \eqref{gen:eq:characteristic-energy-cap} is at most
  \[
    C\|W\|_{C^1}^2\int_{|z-z_p|<\varepsilon}g^{-3}\,dz.
  \]
  Choose a nonzero $r_p\times r_p$ minor of $D_zG(z_p)$.
  As in the rank coordinates of Lemma~\ref{gen:lem:rank-differentiation},
  a local $C^1$ diffeomorphism $z\mapsto(a,b)$ has
  $a\in\RR^{r_p}$ consisting of selected components of $G$ and
  $|G|\ge|a|$. Its Jacobian and inverse Jacobian are bounded.
  The image of $|z-z_p|<\varepsilon$ lies in a box
  $|a|<C\varepsilon$, $|b-b_p|<C\varepsilon$. Thus, for
  $0\le\gamma<r_p$,
  \[
    \int_{|z-z_p|<\varepsilon}g^{-\gamma}\,dz
    \le C\int_{|b-b_p|<C\varepsilon}
          \int_{|a|<C\varepsilon}|a|^{-\gamma}\,da\,db
    \le C_\gamma\varepsilon^{2n-r_p}\varepsilon^{r_p-\gamma}
    =C_\gamma\varepsilon^{2n-\gamma},
  \]
  proving \eqref{gen:eq:characteristic-rank-cap}. Taking $\gamma=3$
  proves \eqref{gen:eq:characteristic-energy-cap}. If $D_zG(z_p)$
  is invertible, the inverse function
  theorem gives $g(z)\asymp|z-z_p|$ near $p$, and the preceding
  pointwise bounds give \eqref{gen:eq:characteristic-energy-pointwise}.
\end{proof}}

\subsection{variations and global stationarity}

Choose charts compatible with the outward boundary orientation
and a consistent sign for their horizontal cofactors.
For a $C^2$ chart $\iota\colon U\to\Sigma$, write
$G_\iota\in\RR^{2n}$ in the fixed frame, so $d\sigma_H=|G_\iota|\,dz$.
For a chart transition $z=\chi(y)$, the cofactor transforms as
$G_{\iota\circ\chi}=(\det D\chi)(G_\iota\circ\chi)$.
At a zero its derivative is therefore
\[
  D_yG_{\iota\circ\chi}
  =(\det D\chi)(D_zG_\iota\circ\chi)D\chi.
\]
The term differentiating the determinant vanishes at the zero.
Since $D\chi$ is invertible, the rank is independent of the chart.

\begin{lemma}[Uniform sublevel estimate]
  \label{gen:rig:lem:parameter-sublevel}
  Let \(n\ge2\), let \(\Sigma\subset\Hn\) be an oriented \(C^2\)
  hypersurface, and let \(\iota\colon U\to\Sigma\) be an oriented
  chart. Let \(f_s\) be a finite-dimensional family with
  \(f_0=\iota\), and let \(G_s\) be its horizontal cofactor.
  We allow the following families:
  \begin{enumerate}
    \item \(f_s=F_s\circ\iota\), where \(F_0=\Id\) and
      \(F(s,p)\) is jointly \(C^3\) in \((s,p)\);
    \item \(f_\tau=\Phi_\tau^X\circ\iota\), where \(X\) is a
      \(C^2\) ambient vector field;
    \item on the chart, \(f_s=\iota+\sum_i s_iW_i\) in the fixed
      coordinates \((z,t)\), where
      \(W_i\in C^{1,1}(U;\RR^{2n+1})\).
  \end{enumerate}
  With \(A_k=\partial_{x_k}f_s\) and
  \(Z_1,\ldots,Z_{2n}=X_1,Y_1,\ldots,X_n,Y_n\), set
  \[
    G_{s,a}=-\det[Z_a(f_s),A_1,\ldots,A_{2n}].
  \]
  For every \(K\Subset U\), there are \(\delta,\rho_0>0\)
  such that, for each \(0\le\gamma<n\), some \(C_\gamma\) satisfies
  \begin{equation}
    \label{gen:rig:eq:uniform-parameter-sublevel}
    \sup_{|s|<\delta}
    \int_{K\cap\{|G_s|<\rho\}}|G_s|^{-\gamma}\,dx
    \le C_\gamma\rho^{n-\gamma},
    \qquad 0<\rho<\rho_0.
  \end{equation}
  The zero set of \(G_s\) is null for every \(|s|<\delta\), and
  \[
    \sup_{|s|<\delta}
    \int_{K\cap\{|G_s|<\rho\}}
      (1+|G_s|^{-1})\,dx
    \le C(\rho^n+\rho^{n-1}).
  \]
\end{lemma}

\begin{proof}
  Choose \(K_1\) with \(K\Subset K_1\Subset U\). At every zero of
  \(G_0\) in \(K_1\), Corollary~\ref{gen:rig:lem:sublevel}
  and the chart-invariance calculation above give
  \(\operatorname{rank}D_xG_0\ge n\).

  For a jointly \(C^3\) ambient family, the cofactor formula shows
  that \(G_s\), its first two parameter derivatives, and \(D_xG_s\)
  are jointly continuous. Lemma~\ref{gen:lem:rank-differentiation},
  applied with \(r=n\) on a neighborhood of \(K\), gives
  \eqref{gen:rig:eq:uniform-parameter-sublevel}.
  The same argument applies to a \(C^2\) flow. With
  \(M_\tau=D\Phi_\tau^X\), its spatial derivative satisfies
  \[
    \begin{aligned}
      \partial_\tau M_\tau
        &=DX(\Phi_\tau^X)M_\tau,\\
      \partial_\tau^2M_\tau
        &=D^2X(\Phi_\tau^X)[X(\Phi_\tau^X),M_\tau]
          +DX(\Phi_\tau^X)^2M_\tau.
    \end{aligned}
  \]
  These expressions are continuous when \(X\in C^2\).

  For an affine family, \(G_s\) is polynomial in \(s\): each
  \(\partial_{x_k}f_s\) is affine in \(s\), as are the
  {coordinate components} of \(Z_a(f_s)\). On a compact
  subchart, the \(C^{1,1}\) regularity of the \(W_i\) consequently gives
  \begin{equation}
    \label{gen:rig:eq:affine-lipschitz-perturbation}
    |G_s(x)-G_0(x)|\le C|s|,
    \qquad
    |(G_s-G_0)(x)-(G_s-G_0)(x')|
      \le C|s|\,|x-x'|.
  \end{equation}
  We use a separate coordinate argument here because \(D_xG_s\)
  need not be continuous for \(C^{1,1}\) fields.

  Fix a zero \(x_0\) of \(G_0\) in \(K\). Relabel coordinates and
  cofactor components so that
  \[
    \det(\partial_{x_j}G_{0,a}(x_0))_{a,j=1}^n\ne0,
    \qquad
    \chi_s(x)
      =(G_{s,1}(x),\ldots,G_{s,n}(x),
        x_{n+1},\ldots,x_{2n}).
  \]
  The inverse function theorem makes \(\chi_0\) a \(C^1\)
  coordinate map on a neighborhood \(B\) of \(x_0\). Shrink \(B\)
  so that \(\chi_0(B)\) is a convex box with compact closure in the
  inverse's domain. Its inverse has a bounded derivative, so
  \[
    |x-x'|\le C|\chi_0(x)-\chi_0(x')|,
    \qquad x,x'\in B.
  \]
  Only the first \(n\) entries of \(\chi_s\) differ from those of
  \(\chi_0\). By \eqref{gen:rig:eq:affine-lipschitz-perturbation},
  \[
    |(\chi_s-\chi_0)(x)-(\chi_s-\chi_0)(x')|
      \le C|s|\,|\chi_0(x)-\chi_0(x')|.
  \]
  For \(C|s|\le\frac12\), the triangle inequality yields
  \[
    \frac12|\chi_0(x)-\chi_0(x')|
      \le|\chi_s(x)-\chi_s(x')|
      \le\frac32|\chi_0(x)-\chi_0(x')|.
  \]
  Thus \(\chi_s\) is uniformly bi-Lipschitz on \(B\).

  Write \(\chi_s(x)=(a,b)\), with \(a,b\in\RR^n\).
  Since \(a\) consists of components of \(G_s\),
  \(|G_s(x)|\ge|a|\). Let \(B_b\subset\RR^n\) be a fixed bounded box
  containing the \(b\) coordinates of \(B\). The inverse Jacobian
  of \(\chi_s\) is uniformly bounded
  almost everywhere. The change-of-variables formula for
  bi-Lipschitz maps therefore gives
  \[
    \begin{aligned}
      \int_{B\cap\{|G_s|<\rho\}}|G_s|^{-\gamma}\,dx
      &\le C\int_{B_b}\int_{|a|<\rho}
                   |a|^{-\gamma}\,da\,db\\
      &=\frac{C|B_b|\,|S^{n-1}|}{n-\gamma}\rho^{n-\gamma}.
    \end{aligned}
  \]
  The same coordinates show that
  \(\{G_s=0\}\cap B\subset\{a=0\}\) has measure zero.
  Finitely many such neighborhoods cover the zeros of \(G_0\)
  in \(K\). On the remaining compact set, \(|G_0|\) has a positive
  lower bound, which persists for small \(s\) by
  \eqref{gen:rig:eq:affine-lipschitz-perturbation}. This proves
  \eqref{gen:rig:eq:uniform-parameter-sublevel} and nullity for the
  affine family. Taking \(\gamma=0\) and \(\gamma=1\), which is
  permitted because \(n\ge2\), gives the final bound.
\end{proof}

\begin{proposition}[Differentiation of horizontal perimeter]
  \label{gen:rig:prop:second-variation}
  Let \(n\ge2\), and let \(\Omega\subset\Hn\) be a bounded open set
  with outward-oriented \(C^2\) boundary \(\Sigma\). For the
  variation families specified in
  Lemma~\ref{gen:rig:lem:parameter-sublevel}, horizontal perimeter
  and enclosed volume are \(C^2\) in the parameters near the identity.
  For every \(h\in\RR\), the affine second variation
  \(\mathcal B_h\) extends uniquely to a continuous bilinear form
  on \(C^1\) surface fields. {For all $C^1$ surface fields
  $U,W$, one has}
  \begin{equation}
    \label{gen:rig:eq:C1-hessian-bound}
    {|\mathcal B_h(U,W)|
      \le C_{\Sigma,h}\|U\|_{C^1}\|W\|_{C^1}.}
  \end{equation}
\end{proposition}

The $C^1$ conclusion is a continuous bilinear extension of the affine
second variation. Twice differentiability along a deformation is
asserted for the families in Lemma~\ref{gen:rig:lem:parameter-sublevel}.

\begin{proof}
  \label{gen:app:parameter-proof}
  The boundary \(\Sigma\) is compact. Choose a finite oriented
  \(C^2\) atlas \(\iota_\ell\colon U_\ell\to\Sigma\) and a
  subordinate partition of unity \(\eta_\ell\), with each
  \(\operatorname{supp}\eta_\ell\Subset U_\ell\). On one chart write
  \(f_s=F_s\circ\iota=(z_s,t_s)\) and
  \(A_k=\partial_{x_k}f_s\). In the affine case \(f_s\) is the
  surface map itself. Define the fixed-chart densities by
  \[
    G_{s,a}=-\det[Z_a(f_s),A_1,\ldots,A_{2n}],
    \qquad v_s=t_s\det D_xz_s.
  \]
  Then \(|G_s|\,dx\) is the perimeter density and
  \(v_s\,dx=f_s^*(t\,dz)\). Since \(d(t\,dz)=dV\), Stokes'
  theorem and the fixed partition give
  \begin{equation}
    \label{gen:eq:general-fixed-atlas-perimeter}
    \PH(F_s\Sigma)
      =\sum_\ell\int_{U_\ell}\eta_\ell|G_{\ell,s}|\,dx,
    \qquad
    V(\Omega_s)
      =\sum_\ell\int_{U_\ell}\eta_\ell v_{\ell,s}\,dx.
  \end{equation}

  Work on one chart and write \(G=G_s\), \(g=|G|\),
  \(G_i=\partial_{s_i}G\), and
  \(G_{ij}=\partial_{s_i}\partial_{s_j}G\).
  On \(\{g>0\}\), set
  \(\widehat G=G/g\) and
  \(\Pi_G=I-\widehat G\otimes\widehat G\).
  Direct differentiation of \(g^2=G\cdot G\) gives
  \(\partial_{s_i}g=(G\cdot G_i)/g\). Differentiating again,
  \begin{equation}
    \label{gen:rig:eq:-second-derivative}
    \begin{aligned}
      \partial_{s_i}\partial_{s_j}g
      &=\frac{G_i\cdot G_j+G\cdot G_{ij}}{g}
        -\frac{(G\cdot G_i)(G\cdot G_j)}{g^3}\\
      &=\frac{\langle\Pi_GG_i,\Pi_GG_j\rangle}{g}
        +\langle\widehat G,G_{ij}\rangle.
    \end{aligned}
  \end{equation}
  This displays the only singular factor, \(g^{-1}\).
  The regularity verified in the proof of
  Lemma~\ref{gen:rig:lem:parameter-sublevel} makes \(G_i,G_{ij}\)
  continuous and bounded on each fixed compact chart.

  For an affine family \(f_s=\iota+\sum_i s_iW_i\),
  {each $Z_a$ is affine in the fixed coordinates}, so \(D^2Z_a=0\).
  Differentiating the determinant in each column gives
  \[
    \begin{aligned}
      G_{i,a}
      &=-\det[DZ_a(f_s)W_i,A_1,\ldots,A_{2n}]\\
      &\quad-\sum_{k=1}^{2n}
        \det[Z_a(f_s),A_1,\ldots,
             \partial_{x_k}W_i,\ldots,A_{2n}],\\
      G_{ij,a}
      &=-\sum_{k=1}^{2n}
        \det[DZ_a(f_s)W_i,A_1,\ldots,
             \partial_{x_k}W_j,\ldots,A_{2n}]\\
      &\quad-\sum_{k=1}^{2n}
        \det[DZ_a(f_s)W_j,A_1,\ldots,
             \partial_{x_k}W_i,\ldots,A_{2n}]\\
      &\quad-\sum_{1\le k<l\le2n}
        \det[Z_a(f_s),A_1,\ldots,
             \partial_{x_k}W_i,\ldots,
             \partial_{x_l}W_j,\ldots,A_{2n}]\\
      &\quad-\sum_{1\le k<l\le2n}
        \det[Z_a(f_s),A_1,\ldots,
             \partial_{x_k}W_j,\ldots,
             \partial_{x_l}W_i,\ldots,A_{2n}].
    \end{aligned}
  \]
  Each replacement occupies its original column. On a fixed compact
  chart, all unreplaced columns are bounded. With
  \(M_i=|W_i|+|D_xW_i|\), the determinant bound gives
  \begin{equation}
    \label{gen:eq:general-affine-cofactor-bounds}
    |G_i|\le CM_i,\qquad |G_{ij}|\le CM_iM_j.
  \end{equation}
  Inserting these
  bounds into \eqref{gen:rig:eq:-second-derivative} gives
  \begin{equation}
    \label{gen:eq:general-affine-density}
    |\partial_{s_i}\partial_{s_j}g|
      \le C(1+g^{-1})M_iM_j
      \quad\text{almost everywhere on }\{g>0\}.
  \end{equation}

  We now justify passing the derivatives through
  \eqref{gen:eq:general-fixed-atlas-perimeter}. Put
  \(g_\varepsilon=(|G|^2+\varepsilon^2)^{1/2}\).
  Its derivatives are continuous, and direct calculation gives
  \[
    \begin{aligned}
      \partial_{s_i}g_\varepsilon
        &=\frac{G\cdot G_i}{g_\varepsilon},\\
      \partial_{s_i}\partial_{s_j}g_\varepsilon
        &=\frac{G_i\cdot G_j+G\cdot G_{ij}}
                {g_\varepsilon}
          -\frac{(G\cdot G_i)(G\cdot G_j)}
                 {g_\varepsilon^3}.
    \end{aligned}
  \]
  Since \(|G|\le g_\varepsilon\), these formulas imply
  \[
    |\partial_{s_i}g_\varepsilon|\le|G_i|,
    \qquad
    |\partial_{s_i}\partial_{s_j}g_\varepsilon|
      \le|G_{ij}|+\frac{2|G_i||G_j|}{g_\varepsilon}.
  \]
  Off the null zero set, the displayed densities are bounded by
  \(C(1+g^{-1})\), uniformly in \(0<\varepsilon\le1\).
  Lemma~\ref{gen:rig:lem:parameter-sublevel} gives
  \[
    \sup_{|s|<\delta}
    \int_{\operatorname{supp}\eta_\ell\cap\{g<\rho\}}
      (1+g^{-1})\,dx
    \le C(\rho^n+\rho^{n-1}).
  \]
  On \(\{g\ge\rho\}\), the differences
  \(D_s^j(g_\varepsilon-g)\) are \(O_\rho(\varepsilon^2)\) for
  \(j=0,1,2\). Splitting the chart integral at \(g=\rho\)
  therefore yields
  \[
    \sup_{|s|<\delta}\left|
      D_s^j\int_{U_\ell}\eta_\ell g_\varepsilon\,dx
       -\int_{U_\ell}\eta_\ell D_s^j g\,dx
    \right|
    \le C_\rho\varepsilon^2+C(\rho^n+\rho^{n-1}),
    \qquad j=0,1,2.
  \]
  {First let \(\varepsilon\downarrow0\), then
  \(\rho\downarrow0\) in this estimate. The regularized chart integrals
  and their first two parameter derivatives converge locally uniformly.
  Since the regularized integrals are \(C^2\), their limit is \(C^2\),
  with derivatives given by those limits. Therefore}
  \[
    D_s^j\PH(F_s\Sigma)
      =\sum_\ell\int_{U_\ell}\eta_\ell
        D_s^j|G_{\ell,s}|\,dx,
    \qquad j=1,2.
  \]

  Volume is simpler because \(v_s=t_s\det D_xz_s\) has no singular
  factor. Its parameter derivatives through order two are
  continuous, so they pass directly through the fixed-chart
  integrals. For the affine family, the product rule and
  multilinearity of the determinant give, at \(s=0\),
  \[
    \begin{aligned}
      \partial_{s_i}\partial_{s_j}v_s\big|_0
      &=W_i^t\sum_{k=1}^{2n}
        \det[\partial_{x_1}z_0,\ldots,
             \partial_{x_k}W_j^z,\ldots,\partial_{x_{2n}}z_0]\\
      &\quad+W_j^t\sum_{k=1}^{2n}
        \det[\partial_{x_1}z_0,\ldots,
             \partial_{x_k}W_i^z,\ldots,\partial_{x_{2n}}z_0]\\
      &\quad+t_0\sum_{k\ne l}
        \det[\partial_{x_1}z_0,\ldots,
             \partial_{x_k}W_i^z,\ldots,
             \partial_{x_l}W_j^z,\ldots,\partial_{x_{2n}}z_0].
    \end{aligned}
  \]
  Again, replacements occupy their original columns. On the fixed
  compact chart, all unreplaced columns and \(t_0\) are bounded.
  Hence
  \[
    \left|\left.\partial_{s_i}\partial_{s_j}v_s\right|_{s=0}\right|
      \le CM_iM_j.
  \]

  Finally, sum \eqref{gen:eq:general-affine-density} and the volume
  estimate over the finite atlas. Lemma~
  \ref{gen:rig:lem:parameter-sublevel} with \(\gamma=1\) gives
  \[
    {\begin{aligned}
      |\mathcal B_h(U,W)|
      &\le C_h\sum_\ell
        \int_{\operatorname{supp}\eta_\ell}
        (1+|G_{\ell,0}|^{-1})
        (|U|+|D_xU|)(|W|+|D_xW|)\,dx\\
      &\le C_{\Sigma,h}\|U\|_{C^1}\|W\|_{C^1}.
    \end{aligned}}
  \]
  {The mixed density is bilinear in \(U,W\) and their first chart
  derivatives. Since restrictions of smooth ambient fields are dense
  in \(C^1\), the bound gives the unique continuous bilinear extension.}
\end{proof}

\Needspace{12\baselineskip}

\begin{samepage}
  Under the assumptions of Proposition~\ref{gen:rig:prop:second-variation},
  fix $h\in\RR$, put $\mathcal F_h=\PH-hV$, and
  let $\mathcal B_h$ be its extended affine second variation in the
  fixed coordinates $(z,t)$.
  \begin{proposition}
    \label{gen:rig:prop:global-stationarity}
    For every $C^1$ surface field $W$,
    \begin{equation}
      \label{gen:eq:general-first-variation}
      \delta\mathcal F_h[W]
      =\int_\Sigma(H_\Sigma-h)\varphi_W\,d\sigma_H,
    \end{equation}
    with an absolutely convergent integral. For the dilation field
    $D=z\cdot\partial_z+2t\partial_t$, one has
    \begin{align}
      \delta V[D]&=(2n+2)|\Omega|,\notag\\*
      \delta\mathcal F_h[D]&=(2n+1)\PH-h\,\delta V[D].
      \label{gen:eq:general-dilation-first}
    \end{align}
    If $H_\Sigma=h$ on $\Sigma^*$, then $\delta\mathcal F_h=0$, and the
     second variation defines the symmetric,
    acceleration-independent Hessian $\mathcal B_h$.
  \end{proposition}
\end{samepage}
\begin{proof}
  By linearity, localize $W$ using the fixed atlas and partition above,
  and put $K=\operatorname{supp}W$ in a characteristic chart.
  Apply \eqref{gen:eq:parameter-rank-caps} at $s=0$ with
  $r=n$ and $\gamma=0$:
  \[
    |K\cap\{g<2\rho\}|
    =\int_{K\cap\{g<2\rho\}}1\,dz
    \le C(2\rho)^n.
  \]
  Choose $\chi_\rho(g)=0$ for $g\le\rho$ and $\chi_\rho(g)=1$
  for $g\ge2\rho$, with $|D_z\chi_\rho|\le C/\rho$.
  The product rule gives
  \[
    D_z((1-\chi_\rho)W)
    =(1-\chi_\rho)D_zW-WD_z\chi_\rho,
  \]
  hence
  \[
    |(1-\chi_\rho)W|+|D_z((1-\chi_\rho)W)|
    \le C\|W\|_{C^1}(1+\rho^{-1})\mathbf1_{\{g<2\rho\}}.
  \]
  {For a $C^1$ surface field $Z$ supported in $K$, let
  $G_s$ and $v_s$ be the cofactor and volume density of the restricted
  affine deformation $f_s=\iota+sZ$.}
  Write $\dot G_Z=\partial_sG_s|_0$ and $\dot v_Z=\partial_sv_s|_0$.
  The determinant derivatives in the preceding proof, evaluated
  at $s=0$ with $W=Z$, give
  \[
    |\dot G_Z|+|\dot v_Z|\le C(|Z|+|D_zZ|).
  \]
  The bound $\bigl||G_s|-|G_0|\bigr|\le|G_s-G_0|$ permits
  differentiation under the integral also for $C^1$ fields.
  Thus \eqref{gen:eq:general-fixed-atlas-perimeter}, with
  $\widehat G_0=G_0/|G_0|$ off the null characteristic set, gives
  \[
    \begin{aligned}
      \delta\mathcal F_h[Z]
      &=\left.\frac{d}{ds}\right|_{s=0}
        \int_K\bigl(|G_s|-hv_s\bigr)\,dz\\
      &=\int_K\bigl(\langle\widehat G_0,\dot G_Z\rangle
        -h\dot v_Z\bigr)\,dz,\\
      |\delta\mathcal F_h[Z]|
      &\le\int_K\bigl(|\dot G_Z|+|h|\,|\dot v_Z|\bigr)\,dz\\
      &\le C_h\int_K(|Z|+|D_zZ|)\,dz.
    \end{aligned}
  \]
  The constant $C_h$ depends only on the fixed chart and $h$.
  Substituting $Z=(1-\chi_\rho)W$ and using the cutoff bound above,
  \[
    \begin{aligned}
      |\delta\mathcal F_h[(1-\chi_\rho)W]|
      &\le C_h\int_K\bigl(|(1-\chi_\rho)W|
        +|D_z((1-\chi_\rho)W)|\bigr)\,dz\\
      &\le C_h\|W\|_{C^1}(1+\rho^{-1})
        \int_{K\cap\{g<2\rho\}}1\,dz\\
      &\le C_h\|W\|_{C^1}(1+\rho^{-1})(2\rho)^n\\
      &\le C_h\rho^{n-1}\|W\|_{C^1},\qquad 0<\rho<1.
    \end{aligned}
  \]
  By linearity this is
  \begin{equation}
    \label{gen:eq:first-cutoff-bound}
    |\delta\mathcal F_h[W]-\delta\mathcal F_h[\chi_\rho W]|
    \le C\rho^{n-1}\|W\|_{C^1}\longrightarrow0.
  \end{equation}
  Moreover $|H_\Sigma-h|=O(g^{-1})$ and
  $|\varphi_W|\le Cg^{-1}\|W\|_{C^1}$, so the right side of
  \eqref{gen:eq:general-first-variation} has absolute graph density at most
  $Cg^{-1}\|W\|_{C^1}\,dz$, which is integrable.
  For fixed $\rho>0$, $C^1$ approximation gives the noncharacteristic
  first variation \cite[Theorem~5.1]{HP} for the cutoff fields.
  Summing the localized identities and then letting $\rho\downarrow0$ gives
  \eqref{gen:eq:general-first-variation}. If $H_\Sigma=h$,
  stationarity makes every  acceleration term vanish.

  The flow generated by $D$ is
  $\delta_{e^\tau}(z,t)=(e^\tau z,e^{2\tau}t)$.
  Its coordinate Jacobian is
  $e^{2n\tau}e^{2\tau}=e^{(2n+2)\tau}$, while each horizontal
  frame vector scales by $e^\tau$. The horizontal cofactor
  determinant therefore scales by
  $e^{(2n+2)\tau}/e^\tau=e^{(2n+1)\tau}$, so
  \[
    \PH(\delta_{e^\tau}\Sigma)=e^{(2n+1)\tau}\PH(\Sigma),
    \qquad
    V(\delta_{e^\tau}\Omega)=e^{(2n+2)\tau}|\Omega|.
  \]
  Differentiating at $\tau=0$ along this flow gives
  \[
    \begin{aligned}
      \delta V[D]
      &=\left.\frac{d}{d\tau}\right|_{\tau=0}
        e^{(2n+2)\tau}|\Omega|=(2n+2)|\Omega|,\\
      \delta\PH[D]
      &=\left.\frac{d}{d\tau}\right|_{\tau=0}
        e^{(2n+1)\tau}\PH(\Sigma)=(2n+1)\PH(\Sigma).
    \end{aligned}
  \]
  Since $\delta\mathcal F_h[D]=\delta\PH[D]-h\,\delta V[D]$,
  this is \eqref{gen:eq:general-dilation-first}.
  {For a $C^2$ surface field $W$, choose a compactly
  supported $C^2$ ambient extension $\widetilde W$, and take the mixed
  derivative of $\mathcal F_h$ along
  $\delta_{e^\tau}\circ(\operatorname{Id}+s\widetilde W)$.}
  In each fixed chart, its cofactor is polynomial in $s$ and smooth
  in $\tau$, with continuous parameter derivatives through order two
  and continuous first spatial derivatives. The rank estimates and
  norm regularization in Proposition~\ref{gen:rig:prop:second-variation}
  therefore justify the mixed density derivative and the
  second-variation chain rule. Homogeneity gives
  $(2n+1)\delta\PH[W]-h(2n+2)\delta V[W]$,
  while the chain rule gives
  $\mathcal B_h(D,W)+\delta\mathcal F_h[(\mathrm dD)W]$.
  Here $\mathrm dD$ is the ordinary coordinate derivative, and
  $(\mathrm dD)W$ is the mixed acceleration of this family.
  Equating the two expressions gives
  \begin{equation}
    \label{gen:eq:general-dilation-mixed}
    \mathcal B_h(D,W)+\delta\mathcal F_h[(\mathrm dD)W]
    =(2n+1)\delta\mathcal F_h[W]-h\,\delta V[W].
  \end{equation}
  The $C^1$ bounds extend this identity to every $C^1$ surface field.
\end{proof}

\Needspace{14\baselineskip}
\subsection{The regular Jacobi pairing}
\label{gen:app:C2-pairing}
Riemannian approximation formulas for horizontal curvature
linearization are given in~\cite{VerzellesiVariation}. On a smooth
regular patch, the horizontal commutators give the following identity.

\begin{lemma}[Curvature linearization and translation Jacobi fields]
  \label{gen:lem:weak-translation}
  \leavevmode
  \begin{enumerate}[label=\textup{(\roman*)},leftmargin=2.2em]
    \item On a smooth noncharacteristic hypersurface, the  derivative
    of horizontal mean curvature under a smooth velocity $W$ is
    \begin{equation}
      \label{gen:eq:general-curvature-linearization}
      \dot H_\Sigma=L\varphi_W+Z_W H_\Sigma,\qquad
      Z_W=W+\varphi_W e_{2n}\in T\Sigma.
    \end{equation}
    \item If $\Sigma$ is $C^2$ and $H_\Sigma=h$ is constant, the normalized
    normal speed $\varphi_K$ of each Heisenberg left-translation generator
    $K$ is a weak Jacobi field:
    \begin{equation}
      \label{gen:eq:weak-translation-jacobi}
      Q_\Sigma(\varphi_K,\psi)=0,\qquad \psi\in C_c^1(\Sigma^*).
    \end{equation}
  \end{enumerate}
\end{lemma}
\begin{proof}
  Choose a smooth local defining function $F$ with $\Sigma=\{F=0\}$
  and $\nabla F$ pointing outward. On a neighborhood where
  $g=|\nabla_HF|>0$, put
  \[
    \nu=\frac{\nabla_HF}{g},\qquad
    P=\Id-\nu\otimes\nu,\qquad
    \alpha=\frac{TF}{g},\qquad \ell=\log g.
  \]
  On $\Sigma$, $\nu=-e_{2n}$ and $J\nu=e_n$.
  Write $D_1,\ldots,D_{2n}=X_1,Y_1,\ldots,X_n,Y_n$,
  and let $J_{ij}$ be the matrix of $J$ in this frame, so
  $[D_i,D_j]=2J_{ij}T$. The ambient horizontal divergence relative
  to Haar measure is denoted by $\operatorname{Div}_H$; it is
  distinct from the surface divergence relative to $d\sigma_H$.
  Set $H=\operatorname{Div}_H\nu$, so $H|_\Sigma=H_\Sigma$,
  and write $\nu H=\sum_i\nu_iD_iH$.

  Let $A_{ij}=D_i\nu_j$, $a=\nabla^\xi_\nu\nu$, and
  $\ell_i=D_i\ell$. Differentiating $D_iF=g\nu_i$ gives
  \[
    A_{ij}-A_{ji}=2\alpha J_{ij}-\ell_i\nu_j+\ell_j\nu_i,
    \qquad A\nu=0,\qquad
    P\nabla_H\ell=a+2\alpha J\nu.
  \]
  Since $[T,D_i]=0$, differentiation of $TF=g\alpha$ also gives
  \[
    T\nu=\nabla_H\alpha+\alpha\nabla_H\ell-\nu T\ell.
  \]
  Applying the horizontal commutators to $a$ and $J\nu$, respectively,
  yields
  \[
    \begin{aligned}
      \operatorname{Div}_H a-\nu H
        &=\tr(A^2)+2\langle J\nu,T\nu\rangle,\\
      \operatorname{Div}_H(J\nu)
        &=2n\alpha-\langle\nabla_H\ell,J\nu\rangle.
    \end{aligned}
  \]
  The terms involving $\nabla_H\ell$ cancel upon substitution, and hence
  \[
    \operatorname{Div}_H(P\nabla_H\ell)-\nu H
      =\tr(A^2)+4e_n\alpha+4n\alpha^2.
  \]
  For $B=PAP$, one has $\tr(A^2)=\tr(B^2)$,
  $B-B^{\mathsf T}=2\alpha PJP$, and $|PJP|^2=2n-2$.
  Its Frobenius norm satisfies $|B|^2=|\II|^2$, so
  \[
    \tr(A^2)=|B|^2-\tfrac12|B-B^{\mathsf T}|^2
       =|\II|^2-4(n-1)\alpha^2.
  \]
  Therefore, on $\Sigma$,
  \begin{equation}
    \label{gen:eq:implicit-jacobi-potential}
    \operatorname{Div}_H(P\nabla_H\log g)-\nu H
      =|\II|^2+4e_n\alpha+4\alpha^2=q_\Sigma.
  \end{equation}

  For a smooth ambient extension of $\psi$, coarea and integration
  by parts give
  \begin{equation}
    \label{gen:eq:implicit-surface-laplacian}
    \Delta_b^\Sigma\psi
      =\frac1g\operatorname{Div}_H(gP\nabla_H\psi)
      \quad\text{on }\Sigma.
  \end{equation}
  Indeed, $d\sigma_H=g\delta(F)\,dV$ in the local coarea formula;
  the derivative of the level-set factor vanishes when paired with
  $P\nabla_H\psi$, since $P\nabla_HF=0$. Equivalently, one may
  replace $\delta$ by smooth one-dimensional approximate identities
  and pass to the limit.

  For a smooth ambient deformation $\Phi_s$ with initial velocity
  $W$, let $F_s=F\circ\Phi_s^{-1}$ and $\lambda=WF$.
  The Eulerian derivatives at zero are
  \[
    \dot F_{\mathrm E}=-\lambda,\qquad
    \dot\nu_{\mathrm E}=-g^{-1}P\nabla_H\lambda,\qquad
    \dot H_{\mathrm E}
       =-\operatorname{Div}_H(g^{-1}P\nabla_H\lambda).
  \]
  Put $\varphi=\lambda/g$, whose restriction to $\Sigma$ is
  $\varphi_W$. By \eqref{gen:eq:implicit-surface-laplacian} and
  \eqref{gen:eq:implicit-jacobi-potential},
  \[
    \begin{aligned}
      \dot H_{\mathrm E}
        &=-\operatorname{Div}_H(P\nabla_H\varphi
                     +\varphi P\nabla_H\log g)\\
        &=-\Delta_b^\Sigma\varphi
                     -\varphi\operatorname{Div}_H(P\nabla_H\log g)\\
        &=L\varphi-\varphi\nu H.
    \end{aligned}
  \]
  Pullback adds $WH$. Since $(W-\varphi\nu)F=0$,
  $W-\varphi\nu=Z_W$ is tangent to $\Sigma$, and
  $\dot H_\Sigma=L\varphi_W+Z_WH_\Sigma$.
  This calculation requires only $g>0$ and includes $\alpha=0$.

  Let $\Psi_s$ be the left-translation flow generated by $K$.
  For every smooth hypersurface $\Sigma_j$, translation invariance gives
  \[
    H_{\Psi_s(\Sigma_j)}\circ\Psi_s=H_{\Sigma_j},
    \qquad
    0=\dot H_{\Sigma_j}
     =L_j\varphi_{K,j}+Z_{K,j}H_{\Sigma_j}.
  \]
  Fix a compact regular $C^2$ patch with $H_\Sigma=h$ and choose
  smooth immersions $\iota_j\to\iota$ in $C^2$ there.
  Pull all quantities back to the same coordinates and transport
  a fixed test $\psi\in C_c^1$ to each patch.
  Write $\rho_j\,dx$ for the perimeter measure and
  $\widehat Z_{K,j}$ for the coordinate representation of $Z_{K,j}$.
  On a fixed compact coordinate set $K_0$ containing the test support,
  regularity and $C^2$ convergence give
  \[
    \begin{gathered}
      \inf_{K_0}\rho_j\ge c>0,\qquad
      \rho_j\to\rho,\quad
      \varphi_{K,j}\to\varphi_K,\quad
      \widehat Z_{K,j}\to\widehat Z_K\quad\hbox{in }C^1,\\
      H_{\Sigma_j}\to h\quad\hbox{uniformly}.
    \end{gathered}
  \]
  Since $Z_{K,j}h=0$, integration by parts gives
  \[
    \begin{aligned}
      Q_{\Sigma_j}(\varphi_{K,j},\psi)
      &=\int_{\Sigma_j}(L_j\varphi_{K,j})\psi\,d\sigma_{H,j}\\
      &=-\int_{\Sigma_j}\psi Z_{K,j}H_{\Sigma_j}\,d\sigma_{H,j}\\
      &=-\int_{\Sigma_j}\psi Z_{K,j}(H_{\Sigma_j}-h)\,d\sigma_{H,j}\\
      &=\int_{\Sigma_j}(H_{\Sigma_j}-h)
           \operatorname{div}_{H,j}(\psi Z_{K,j})\,d\sigma_{H,j}.
    \end{aligned}
  \]
  There is no boundary term because $\psi$ has compact support.
  Here $\operatorname{div}_{H,j}$ is the divergence on $T\Sigma_j$
  relative to $d\sigma_{H,j}=\rho_j\,dx$, so one takes the ordinary
  coordinate divergence of the weighted field
  $\rho_j\psi\widehat Z_{K,j}$ and divides by the density $\rho_j$:
  \[
    \begin{aligned}
      (\operatorname{div}_{H,j}(\psi Z_{K,j}))\circ\iota_j
      &=\rho_j^{-1}\operatorname{div}_x(\rho_j\psi\widehat Z_{K,j})\\
      &=\rho_j^{-1}\sum_{a=1}^{2n}
          \partial_{x_a}(\rho_j\psi\widehat Z_{K,j}^{a})\\
      &=\sum_{a=1}^{2n}\widehat Z_{K,j}^{a}\partial_{x_a}\psi
        +\psi\sum_{a=1}^{2n}\left(
          \partial_{x_a}\widehat Z_{K,j}^{a}
          +\rho_j^{-1}(\partial_{x_a}\rho_j)\widehat Z_{K,j}^{a}
        \right).
    \end{aligned}
  \]
  Hence
  \[
    \begin{aligned}
      \|\operatorname{div}_{H,j}(\psi Z_{K,j})\|_{L^\infty(K_0)}
      &\le C\|\psi\|_{C^1(K_0)},\\
      |Q_{\Sigma_j}(\varphi_{K,j},\psi)|
      &\le\|H_{\Sigma_j}-h\|_{L^\infty(K_0)}
        \|\operatorname{div}_{H,j}(\psi Z_{K,j})\|_{L^\infty(K_0)}
        \int_{K_0}\rho_j\,dx\\
      &\le C\|H_{\Sigma_j}-h\|_{L^\infty(K_0)}\|\psi\|_{C^1(K_0)}
       \longrightarrow0.
    \end{aligned}
  \]
  In the same coordinates the Jacobi form is
  \[
    Q_{\Sigma_j}(\varphi_{K,j},\psi)
    =\int_{K_0}\left(\sum_{a,b=1}^{2n}a_j^{ab}
        \partial_{x_a}\varphi_{K,j}\,\partial_{x_b}\psi
        -q_j\varphi_{K,j}\psi\right)\rho_j\,dx,
  \]
  where {$(a_j^{ab})$ is the matrix of the horizontal gradient
  pairing} and $q_j=q_{\Sigma_j}\circ\iota_j$.
  {$a_j^{ab},q_j,\rho_j$ converge uniformly}, and
  $\varphi_{K,j}\to\varphi_K$ in $C^1$. Therefore
  \[
    Q_\Sigma(\varphi_K,\psi)
    =\lim_{j\to\infty}Q_{\Sigma_j}(\varphi_{K,j},\psi)=0.
  \]
  Partitioning the test support into regular coordinate patches gives
  \eqref{gen:eq:weak-translation-jacobi}, including across \mbox{$\alpha=0$}.
\end{proof}

For the local pairing, let $\Sigma$ be a compact oriented $C^2$
hypersurface bounding a domain, and fix a compact set $\mathcal K\subset\Sigma^*$
with $H_\Sigma=h$ near $\mathcal K$. Variations below agree with the original
immersion outside $\mathcal K$ for all sufficiently small parameters.
{The local CMC assumption removes the term involving
$H_\Sigma-h$ from the affine Hessian and is needed for the exact
Hessian--Jacobi pairing below.}

\begin{proposition}
  \label{gen:prop:regularvariation}
  {The affine Hessian extends continuously to $C^1$ surface fields}
  $W_1,W_2$ supported in $\mathcal K$ and satisfies
  \begin{equation}
    \label{gen:eq:C2-local-pairing}
    \mathcal B_h(W_1,W_2)=Q_\Sigma(\varphi_{W_1},\varphi_{W_2}).
  \end{equation}
\end{proposition}
Consequently, for the variation classes in
Proposition~\ref{gen:rig:prop:second-variation} with the support condition above,
$\delta(\PH-hV)[W]=0$ and
\begin{equation}
  \label{gen:eq:regular-hessian}
  (\PH-hV)''(0)=Q_\Sigma[\varphi].
\end{equation}
\begin{proof}
  Put $\mathcal F_h=\PH-hV$. {For every $C^1$ surface field
  $W$ supported in $\mathcal K$,}
  the first variation gives
  \[
    \delta\mathcal F_h[W]
    =\int_{\mathcal K}(H_\Sigma-h)\varphi_W\,d\sigma_H=0.
  \]
  Work first on a compact regular coordinate patch and choose smooth
  immersions $\iota_j\to\iota$ in $C^2$ and smooth fields
  $W_{i,j}\to W_i$ in $C^1$, with common compact support for the fields.
  All quantities below are pulled back to the same coordinates.
  For the affine family
  \[
    \iota_{j,s,t}=\iota_j+sW_{1,j}+tW_{2,j},
    \qquad
    B_j=\left.\partial_s\partial_t\right|_{(0,0)}
          \mathcal F_h(\iota_{j,s,t}),
  \]
  write $\rho_{j,s}\,dx$ for the perimeter measure at $t=0$ and
  $\varphi_{2,j,s}$ for the normalized speed of $W_{2,j}$ there.
  Omit $s$ at zero and use dots for $\partial_s|_0$.
  The product rule gives
  \[
    \left.\partial_s(\rho_{j,s}\varphi_{2,j,s})\right|_0
    =\dot\rho_j\varphi_{2,j}+\rho_j\dot\varphi_{2,j}.
  \]
  With $Z_{i,j}=W_{i,j}+\varphi_{i,j}e_{2n,j}$ for $i=1,2$,
  \eqref{gen:eq:general-curvature-linearization} gives
  $\dot H_j=L_j\varphi_{1,j}+Z_{1,j}H_j$. Thus
  \[
    \begin{aligned}
      B_j
      &=\left.\partial_s\right|_0
        \int(H_{j,s}-h)\varphi_{2,j,s}\rho_{j,s}\,dx\\
      &=\int\dot H_j\varphi_{2,j}\rho_j\,dx
        +\int(H_j-h)\left.\partial_s(\rho_{j,s}\varphi_{2,j,s})\right|_0dx\\
      &=\int(L_j\varphi_{1,j})\varphi_{2,j}\rho_j\,dx
        +\int\varphi_{2,j}Z_{1,j}H_j\,\rho_j\,dx\\
      &\quad+\int(H_j-h)\left.\partial_s
                       (\rho_{j,s}\varphi_{2,j,s})\right|_0dx.
    \end{aligned}
  \]
  Integration by parts in the first term gives
  \[
    \begin{aligned}
      \int(L_j\varphi_{1,j})\varphi_{2,j}\,d\sigma_{H,j}
      &=\int\bigl(\langle\nabla_b\varphi_{1,j},\nabla_b\varphi_{2,j}\rangle
           -q_j\varphi_{1,j}\varphi_{2,j}\bigr)\,d\sigma_{H,j}\\
      &=Q_{\Sigma_j}(\varphi_{1,j},\varphi_{2,j}).
    \end{aligned}
  \]
  Since $h$ is constant, the second term becomes
  \[
    \begin{aligned}
      \int\varphi_{2,j}Z_{1,j}H_j\,d\sigma_{H,j}
      &=\int\varphi_{2,j}Z_{1,j}(H_j-h)\,d\sigma_{H,j}\\
      &=-\int(H_j-h)
          \operatorname{div}_{H,j}(\varphi_{2,j}Z_{1,j})\,d\sigma_{H,j}.
    \end{aligned}
  \]
  Both boundary terms vanish by compact support. Here the divergence
  is taken on the full tangent bundle with respect to $d\sigma_{H,j}$.
  Combining the three terms yields
  \[
    \begin{aligned}
      B_j&=Q_{\Sigma_j}(\varphi_{1,j},\varphi_{2,j})
        +\int(H_j-h)R_j\,d\sigma_{H,j},\\
      R_j&=\left.\rho_j^{-1}\partial_s
        (\rho_{j,s}\varphi_{2,j,s})\right|_0
        -\operatorname{div}_{H,j}(\varphi_{2,j}Z_{1,j}).
    \end{aligned}
  \]

  To estimate the first term of $R_j$ directly, put
  $A_{k,j}=\partial_{x_k}\iota_j$ and
  $A_{k,j,s}=A_{k,j}+s\partial_{x_k}W_{1,j}$.
  In the outward-oriented chart, the normal flux identity is
  \[
    \varphi_{2,j,s}\,d\sigma_{H,j,s}
    =\langle W_{2,j},N_{j,s}\rangle\,d\sigma_{j,s}
    =\iota_{j,s,0}^*(\iota_{W_{2,j}}dV).
  \]
  Since $dV=dz\,dt$ and $W_{2,j}$ is independent of $s$ in the
  fixed ambient coordinates, this gives
  \[
    \begin{aligned}
      \rho_{j,s}\varphi_{2,j,s}
      &=\det[W_{2,j},A_{1,j,s},\ldots,A_{2n,j,s}],\\
      \left.\partial_s(\rho_{j,s}\varphi_{2,j,s})\right|_0
      &=\sum_{k=1}^{2n}\det[W_{2,j},A_{1,j},\ldots,
          \partial_{x_k}W_{1,j},\ldots,A_{2n,j}].
    \end{aligned}
  \]
  In each summand $\partial_{x_k}W_{1,j}$ replaces $A_{k,j}$.
  If $\widehat Z_{1,j}$ denotes the coordinate vector field of $Z_{1,j}$,
  the other term is
  \[
    \begin{aligned}
      \operatorname{div}_{H,j}(\varphi_{2,j}Z_{1,j})
      &=\rho_j^{-1}\sum_{a=1}^{2n}
          \partial_{x_a}(\rho_j\varphi_{2,j}\widehat Z_{1,j}^{a})\\
      &=\sum_{a=1}^{2n}\widehat Z_{1,j}^{a}\partial_{x_a}\varphi_{2,j}
        +\varphi_{2,j}\sum_{a=1}^{2n}
          \left(\partial_{x_a}\widehat Z_{1,j}^{a}
            +\rho_j^{-1}(\partial_{x_a}\rho_j)\widehat Z_{1,j}^{a}\right).
    \end{aligned}
  \]
  Fix a compact coordinate set $K_0$ containing the supports; all norms
  in the following estimates are taken on $K_0$.
  Since $\iota_j\to\iota$ in $C^2$, the cofactor densities satisfy
  $\rho_j\to\rho$ in $C^1$. The patch is noncharacteristic, so
  $m=\min_{K_0}\rho>0$. For all sufficiently large $j$,
  \[
    \rho_j(x)\ge\rho(x)-\|\rho_j-\rho\|_\infty
      \ge\frac m2=:c>0,
    \qquad
    \|\rho_j\|_{C^1}+\|\rho_j^{-1}\|_{C^1}\le C.
  \]
  To see which derivatives are needed, the same flux identity gives
  \[
    \begin{aligned}
      \varphi_{i,j}
      &=\frac{\det[W_{i,j},A_{1,j},\ldots,A_{2n,j}]}{\rho_j},\\
      \partial_{x_a}\varphi_{i,j}
      &=\frac{\partial_{x_a}\det[W_{i,j},A_{1,j},\ldots,A_{2n,j}]}{\rho_j}
        -\frac{\det[W_{i,j},A_{1,j},\ldots,A_{2n,j}]}{\rho_j^2}
           \partial_{x_a}\rho_j.
    \end{aligned}
  \]
  Differentiating this determinant uses only $D_xW_{i,j}$ and
  $D_xA_{k,j}=D_x\partial_{x_k}\iota_j$.
  The cofactor formula for $e_{2n,j}$ likewise uses only
  $\iota_j,D_x\iota_j$ and division by $\rho_j$; hence
  \[
    \|e_{2n,j}\|_{C^1}\le C,
    \qquad
    \|\varphi_{i,j}\|_{C^1}\le C\|W_{i,j}\|_{C^1},
    \qquad
    \|Z_{i,j}\|_{C^1}\le C\|W_{i,j}\|_{C^1}.
  \]
  {The chart components of $Z_{i,j}$ satisfy the same
  bound.} Indeed, with $A_j=D_x\iota_j$ and ordinary coordinate
  matrix transpose,
  \[
    A_j\widehat Z_{i,j}=Z_{i,j},
    \qquad
    \widehat Z_{i,j}=(A_j^{\mathsf T}A_j)^{-1}A_j^{\mathsf T}Z_{i,j}.
  \]
  The matrices $(A_j^{\mathsf T}A_j)^{-1}A_j^{\mathsf T}$ are uniformly
  bounded in $C^1$, since $A_j\to D_x\iota$ in $C^1$ and the limiting
  immersion has full rank on $K_0$. Therefore
  $\|\widehat Z_{i,j}\|_{C^1}\le C\|W_{i,j}\|_{C^1}$.
  Substituting these bounds in the expanded divergence gives
  \[
    \begin{aligned}
      \|\operatorname{div}_{H,j}(\varphi_{2,j}Z_{1,j})\|_\infty
      &\le C\bigl(\|\widehat Z_{1,j}\|_\infty\|D_x\varphi_{2,j}\|_\infty
        +\|\varphi_{2,j}\|_\infty\|D_x\widehat Z_{1,j}\|_\infty\bigr)\\
      &\quad+C c^{-1}\|D_x\rho_j\|_\infty
                  \|\varphi_{2,j}\|_\infty\|\widehat Z_{1,j}\|_\infty\\
      &\le C\|W_{1,j}\|_{C^1}\|W_{2,j}\|_{C^1}.
    \end{aligned}
  \]
  Combining this with the determinant derivative above,
  \[
    \begin{aligned}
      \left\|\left.\partial_s(\rho_{j,s}\varphi_{2,j,s})\right|_0
        \right\|_\infty
      &\le C\|W_{2,j}\|_\infty\|D_xW_{1,j}\|_\infty,\\
      \|R_j\|_\infty
      &\le c^{-1}\left\|\left.\partial_s
            (\rho_{j,s}\varphi_{2,j,s})\right|_0\right\|_\infty
        +\|\operatorname{div}_{H,j}(\varphi_{2,j}Z_{1,j})\|_\infty\\
      &\le C\|W_{1,j}\|_{C^1}\|W_{2,j}\|_{C^1}\le C.
    \end{aligned}
  \]
  The last constant is independent of $j$, because
  $W_{i,j}\to W_i$ in $C^1$ makes both field norms uniformly bounded.
  In particular, no derivative of $H_j-h$ remains. Since
  $H_j\to h$ uniformly on $K_0$,
  \[
    \begin{aligned}
      |B_j-Q_{\Sigma_j}(\varphi_{1,j},\varphi_{2,j})|
      &\le\|H_j-h\|_{L^\infty(K_0)}\|R_j\|_{L^\infty(K_0)}
            \int_{K_0}\rho_j\,dx\\
      &\le C\|H_j-h\|_{L^\infty(K_0)}\longrightarrow0.
    \end{aligned}
  \]
  The approximating hypersurfaces need not have constant mean curvature.

  The speeds converge in $C^1$, {and $a_j^{ab}$, $q_j$,
  and $\rho_j$ converge uniformly.} Hence
  \[
    \begin{aligned}
      Q_{\Sigma_j}(\varphi_{1,j},\varphi_{2,j})
      &=\int_{K_0}\left(\sum_{a,b=1}^{2n}a_j^{ab}
          \partial_{x_a}\varphi_{1,j}\,\partial_{x_b}\varphi_{2,j}
          -q_j\varphi_{1,j}\varphi_{2,j}\right)\rho_j\,dx\\
      &\longrightarrow Q_\Sigma(\varphi_{W_1},\varphi_{W_2}).
    \end{aligned}
  \]
  {To pass to the limit in $B_j$,} let $G_j(s,t,x)$ and
  $v_j(s,t,x)$ be the
  cofactor and volume densities of $\iota_{j,s,t}$.
  Their determinant formulas show that
  $G_j,\partial_sG_j,\partial_tG_j,\partial_s\partial_tG_j$ and
  $\partial_s\partial_tv_j$, evaluated at $(0,0)$, converge uniformly
  under the stated approximations. With
  $\widehat G_j=G_j/\rho_j$ and
  $\Pi_j=I-\widehat G_j\otimes\widehat G_j$ at $(0,0)$,
  \eqref{gen:rig:eq:-second-derivative} gives
  \[
    \begin{aligned}
      B_j=\int_{K_0}\biggl(
        &\frac{\langle\Pi_j\partial_sG_j,\Pi_j\partial_tG_j\rangle}{\rho_j}
         +\langle\widehat G_j,\partial_s\partial_tG_j\rangle
         -h\partial_s\partial_tv_j\biggr)\,dx
      \longrightarrow\mathcal B_h(W_1,W_2).
    \end{aligned}
  \]
  The denominator is bounded below by $c$, so this limit is also
  uniform at the level of the integrand. The same formula gives
  $|\mathcal B_h(W_1,W_2)|\le C\|W_1\|_{C^1}\|W_2\|_{C^1}$.
  Together with the preceding limits, this proves
  \eqref{gen:eq:C2-local-pairing} on one regular patch and its continuous
  extension to $C^1$ fields, including on patches crossing $\alpha=0$.

  Choose a finite partition $\sum_\ell\eta_\ell=1$ near $\mathcal K$
  subordinate to regular patches. In each patch choose $\zeta_\ell=1$
  near $\operatorname{supp}(\eta_\ell W_1)$, with compact support in that
  patch. Since $\varphi_{\eta_\ell W_1}=\eta_\ell\varphi_{W_1}$,
  bilinearity and locality give
  \[
    \begin{aligned}
      \mathcal B_h(W_1,W_2)
      &=\sum_\ell\mathcal B_h(\eta_\ell W_1,\zeta_\ell W_2)\\
      &=\sum_\ell Q_\Sigma(\eta_\ell\varphi_{W_1},
                           \zeta_\ell\varphi_{W_2})\\
      &=\sum_\ell Q_\Sigma(\eta_\ell\varphi_{W_1},\varphi_{W_2})
       =Q_\Sigma(\varphi_{W_1},\varphi_{W_2}).
    \end{aligned}
  \]
  Finally, for a variation with initial velocity $W$ and 
  acceleration $A$ in the fixed ambient coordinates, the 
  second-derivative formula and stationarity give
  \[
    \mathcal F_h''(0)
    =\mathcal B_h(W,W)+\delta\mathcal F_h[A]
    =\mathcal B_h(W,W)=Q_\Sigma[\varphi_W].
  \]
  This proves \eqref{gen:eq:regular-hessian} for the stated variation classes.
\end{proof}

{\begin{proof}[Proof of Theorem~\ref{thm:c2-global-hessian}]
  The three variation classes are those of
  Lemma~\ref{gen:rig:lem:parameter-sublevel}.
  Proposition~\ref{gen:rig:prop:second-variation} gives their $C^2$
  perimeter and volume functions, the material density formulas,
  and the unique continuous $C^1$ affine extension with the stated bound.
  Proposition~\ref{gen:rig:prop:global-stationarity} gives
  stationarity and acceleration independence when $H_\Sigma=h$.
\end{proof}}

{\begin{proposition}[Global Jacobi pairing under a rank condition]
  \label{gen:prop:rank-global-jacobi}
  Let $\Omega\subset\Hn$ be bounded with compact $C^2$ boundary
  $\Sigma$, and assume $H_\Sigma=h$ on $\Sigma^*$. Near each
  characteristic point $p$, write $\Sigma$ as $t=f_p(z)$ and put
  $G_p=-\nabla_H(t-f_p)$ in the fixed horizontal frame. Suppose
  $\operatorname{rank}D_zG_p(p)\ge r$ at every such point for a
  fixed integer $4\le r\le2n$. Then $Q_\Sigma$ is absolutely
  convergent on normalized speeds of $C^1$ surface fields, and
  \begin{equation}
    \label{eq:abstract-global-hessian}
    \mathcal B_h(W_1,W_2)
    =Q_\Sigma(\varphi_{W_1},\varphi_{W_2})
  \end{equation}
  for all $C^1$ surface fields $W_1,W_2$. For a jointly $C^3$
  ambient diffeomorphism family or a flow of a $C^2$ ambient field
  with initial velocity $W$, one has
  $\mathcal F_h''(0)=Q_\Sigma[\varphi_W]$.
  In particular, these conclusions hold for every compact $C^2$
  CMC boundary when $n\ge4$.
\end{proposition}}

{\begin{proof}
  Proposition~\ref{gen:rig:prop:second-variation} gives the continuous
  affine Hessian, and Proposition~\ref{gen:rig:prop:global-stationarity}
  gives stationarity. Put $\kappa=|N_b|$. This function is Lipschitz
  on $\Sigma$ and $C^1$ on $\Sigma^*$. In a vertical graph chart,
  with $g=|G_p|$, one has
  \[
    \kappa=\frac{g}{\sqrt{1+g^2}},\qquad
    d\sigma_\Sigma=\sqrt{1+g^2}\,dz,\qquad
    d\sigma_H=\kappa\,d\sigma_\Sigma.
  \]
  Cover the compact characteristic set by finitely many rank
  charts from Lemma~\ref{gen:lem:rank-differentiation}. In each
  chart $g$ is comparable to $\kappa$, while $\kappa$ has a positive
  lower bound on the remaining compact set. Thus
  \eqref{gen:eq:parameter-rank-caps} gives
  \begin{equation}
    \label{gen:eq:global-kappa-sublevel}
    \int_{\{0<\kappa<\rho\}}\kappa^{-\gamma}\,d\sigma_\Sigma
      \le C_\gamma\rho^{r-\gamma},\qquad 0\le\gamma<r,
  \end{equation}
  for sufficiently small $\rho>0$. The characteristic set has
  zero surface measure. If it is empty, all the following cutoffs
  equal one for small $\rho$.

  The pointwise graph estimates in the proof of
  Proposition~\ref{gen:prop:characteristic-energy}, together with
  compactness of the regular complement, give
  \[
    |\varphi_W|\le C\kappa^{-1}\|W\|_{C^1},\qquad
    |\nabla_b\varphi_W|\le C\kappa^{-2}\|W\|_{C^1},\qquad
    |q_\Sigma|\le C\kappa^{-2}.
  \]
  In particular, \eqref{gen:eq:global-kappa-sublevel} with
  $\gamma=3$ bounds the energy over $\{0<\kappa<\rho\}$ by
  $C\rho^{r-3}\|W\|_{C^1}^2$.
  Choose a smooth function $0\le\eta\le1$, equal to zero on
  $[0,1]$ and to one on $[2,\infty)$, and set
  $\chi_\rho=\eta(\kappa/\rho)$. Then $\chi_\rho\in C^1(\Sigma)$
  has compact support in $\Sigma^*$ and
  $|\nabla_\Sigma\chi_\rho|\le C\rho^{-1}$.
  Writing $\varphi_i=\varphi_{W_i}$, the cutoff derivative contributes at most
  \[
    C\rho^{-2}\int_{\{0<\kappa<2\rho\}}\varphi_i^2\,d\sigma_H
      \le C\rho^{-2}\|W_i\|_{C^1}^2
          \int_{\{0<\kappa<2\rho\}}\kappa^{-1}\,d\sigma_\Sigma
      \le C\rho^{r-3}\|W_i\|_{C^1}^2.
  \]
  The product rule consequently yields
  \begin{equation}
    \label{gen:eq:abstract-energy-cutoff}
    \int_{\Sigma^*}\left(
       |\nabla_b((1-\chi_\rho)\varphi_i)|^2
       +(1+|q_\Sigma|)|(1-\chi_\rho)\varphi_i|^2
    \right)d\sigma_H
    \le C\rho^{r-3}\|W_i\|_{C^1}^2.
  \end{equation}
  Since $r>3$, this energy tends to zero, and
  $Q_\Sigma(\varphi_1,\varphi_2)$ is absolutely convergent.

  By \eqref{gen:eq:general-affine-density}, the mixed affine
  perimeter density is bounded by $C(1+\kappa^{-1})M_1M_2$
  relative to $d\sigma_\Sigma$, where
  $M_i=|W_i|+|D_zW_i|$ in each fixed chart. The volume density
  is bounded by $CM_1M_2$. The fields $\chi_\rho W_i$ and
  $(1-\chi_\rho)W_i$ have first derivatives bounded by
  $C\rho^{-1}\|W_i\|_{C^1}$, and their contribution to the
  difference of the Hessians is supported in $\{\kappa<2\rho\}$.
  Bilinearity, locality, and \eqref{gen:eq:global-kappa-sublevel}
  with $\gamma=0,1$ therefore give
  \begin{equation}
    \label{gen:eq:abstract-hessian-cutoff}
    \begin{aligned}
      |\mathcal B_h(W_1,W_2)
        -\mathcal B_h(\chi_\rho W_1,\chi_\rho W_2)|
      &\le C\rho^{-2}\|W_1\|_{C^1}\|W_2\|_{C^1}
         \int_{\{0<\kappa<2\rho\}}(1+\kappa^{-1})\,d\sigma_\Sigma\\
      &\le C\rho^{r-3}\|W_1\|_{C^1}\|W_2\|_{C^1}.
    \end{aligned}
  \end{equation}
  The density argument in Proposition~\ref{gen:rig:prop:second-variation}
  extends this bound to $C^1$ fields. By
  Proposition~\ref{gen:prop:regularvariation},
  \[
    \mathcal B_h(\chi_\rho W_1,\chi_\rho W_2)
      =Q_\Sigma(\chi_\rho\varphi_1,\chi_\rho\varphi_2).
  \]
  Letting $\rho\downarrow0$ in
  \eqref{gen:eq:abstract-energy-cutoff} and
  \eqref{gen:eq:abstract-hessian-cutoff} proves
  \eqref{eq:abstract-global-hessian}.
  For the stated variation families, the second-derivative chain rule is
  $\mathcal F_h''(0)=\mathcal B_h(W,W)+\delta\mathcal F_h[A]$,
  where $A$ is the acceleration field. Stationarity removes the
  last term. Finally, the automatic rank bound
  $\operatorname{rank}D_zG_p(p)\ge n$ permits $r=n$ when $n\ge4$.
\end{proof}}

\section{Model geometry and the regular Jacobi form}
\label{sec:regular}

\subsection{Angular conventions}
On the unit sphere $\Sphere\subset\CC^n$, let $\zeta$ be the position
vector and $d\omega$ the round measure. Put
{$\vartheta=-\Im\ip\zeta{d\zeta}$} and
$\CRD=\ker\vartheta$; its complex
structure is multiplication by $i$, and $\gSphere$ is its round metric.
The Hopf field is $\TSphere|_\zeta=i\zeta$.

{The $\CRD$-valued gradient $\gradSphere\psi$ is defined by
\[
  \gSphere(\gradSphere\psi,Y)=Y\psi\qquad(Y\in\CRD).
\]
The corresponding nonpositive CR sub-Laplacian is
\begin{equation}
  \label{eq:angular-operator-definition}
  \DeltaSphere\psi=\operatorname{div}_{d\omega}(\gradSphere\psi).
\end{equation}}

\subsection{General rotational hypersurfaces}
\label{subsec:revolution}
Let $I_f\subset\RR$ be an open interval and let $f\in C^2(I_f)$ be positive.
Consider the rotational hypersurface
\[
  \Sigma_f=\{(f(t)\omega,t):t\in I_f,\ \omega\in\Sphere\},
\]
oriented by the normal pointing toward increasing $|z|$.
Fix $t_*\in I_f$ and define the phase $U$ and horizontal arclength $u$ by
\[
  U(t)=\int_{t_*}^t\frac{ds}{f(s)^2},\qquad
  u(t)=\int_{t_*}^t\sqrt{f'(s)^2+f(s)^{-2}}\,ds.
\]
Since $u'>0$, we may use $u$ in place of $t$. Set
\[
  \Psi_f(t,\zeta)=\bigl(f(t)e^{-iU(t)}\zeta,t\bigr).
\]
The curves with $\zeta$ fixed are called the $u$-curves;
we write $\partial_u=(u')^{-1}\partial_t$.
Let $M=I_f\times\Sphere$, and put
\[
  HTM=\RR\partial_t\oplus\CRD,\qquad
  HT\Sigma_f=\xi\cap T\Sigma_f,\qquad
  J_M(a\partial_t+Y)=iY\quad(Y\in\CRD).
\]

\begin{theorem}
  \label{thm:cr-isomorphism}
  The hypersurface $\Sigma_f$ is noncharacteristic, and $\Psi_f$ is a
  $C^2$ diffeomorphism identifying the horizontal structures
  \[
    \Psi_f:(M,HTM,J_M)\longrightarrow(\Sigma_f,HT\Sigma_f,J').
  \]
  More precisely,
  \[
    d\Psi_f(HTM)=HT\Sigma_f,\qquad
    d\Psi_f\circ J_M=J'\circ d\Psi_f,
  \]
  \[
    d\Psi_f(\partial_u)=-e_n,\qquad d\Psi_f(\CRD)=\xi'.
  \]
  The contact form, horizontal metric and perimeter measure satisfy
  {
  \[
    \Psi_f^*\Theta=f^2\vartheta,\qquad
    \left.\Psi_f^*\LTheta\right|_{HTM}
      =du^2+f^2\left.\gSphere\right|_{\CRD},\qquad
    \Psi_f^*d\sigma_H=f^{2n}\,du\,d\omega.
  \]}
  In particular, for each fixed $t$, the map
  $\zeta\mapsto\Psi_f(t,\zeta)$ is a CR diffeomorphism from
  $(\Sphere,\CRD,i)$ onto the slice $\Sigma_f\cap\{t=\mathrm{constant}\}$
  with its induced CR structure $(\xi',J)$; its horizontal metric is
  $f(t)^2\gSphere|_{\CRD}$.
\end{theorem}

\begin{proof}
  {
  Write $\omega=e^{-iU(t)}\zeta$ and
  \[
    V:=d\Psi_f(\partial_t)=\bigl((f'-i/f)\omega,1\bigr),\qquad
    d\Psi_f(Y)=\bigl(fe^{-iU}Y,0\bigr).
  \]
  The inverse is $(z,t)\mapsto(t,e^{iU(t)}z/f(t))$. Moreover,
  \[
    {\begin{aligned}
      \Theta(V)&=1-\Im\bigl(f(f'+i/f)\bigr)=0,\\
      \Theta(d\Psi_f(Y))&=-f^2\Im\langle\zeta,Y\rangle
                         =f^2\vartheta(Y).
    \end{aligned}}
  \]
  Thus $\Psi_f^*\Theta=f^2\vartheta$ and
  $d\Psi_f(HTM)=HT\Sigma_f$. If $Y\in\CRD$, then
  $\langle\zeta,Y\rangle=0$, so $V$ and $d\Psi_f(Y)$ are orthogonal and
  \[
    |d\Psi_f(a\partial_t+Y)|_{\LTheta}^2
    =a^2(f'^2+f^{-2})+f^2|Y|^2
    =(u')^2a^2+f^2|Y|^2.
  \]
  This proves the metric formula and $|V|_{\LTheta}=u'$.

  The defining function $g(z,t)=|z|-f(t)$ is positive on the side
  toward which the chosen normal points. For a horizontal
  $W=(w,\tau)$, {using $\tau=\Im\langle z,w\rangle$} gives
  \[
    dg(W)=\Re\langle\omega,w\rangle-f'\tau
          =\Re\langle(1+iff')\omega,w\rangle
          =\LTheta(fJV,W).
  \]
  Hence
  \[
    \nabla_Hg=fJV,\qquad |\nabla_Hg|_{\LTheta}=fu'>0,
  \]
  proving noncharacteristicity. Since $e_{2n}$ is the inward horizontal
  unit normal,
  \[
    e_{2n}=-\frac{JV}{u'},\qquad
    e_n=-Je_{2n}=-\frac{V}{u'}=-d\Psi_f(\partial_u).
  \]

  For $Y\in\CRD$, one has $d\Psi_f(iY)=Jd\Psi_f(Y)$.
  The $J$-invariant space $d\Psi_f(\CRD)$ has dimension $2n-2$, hence
  equals $\xi'$. Together with $V=-u'e_n$, this proves the asserted
  intertwining relation and the CR statement on each slice.

  Finally, for the rotational graph $(t,\omega)\mapsto(f(t)\omega,t)$,
  \[
    d\sigma_E=f^{2n-1}\sqrt{1+f'^2}\,dt\,d\omega,\qquad
    |\nabla_Eg|=\sqrt{1+f'^2}.
  \]
  Since the phase rotation preserves $d\omega$ and
  $d\sigma_H=(|\nabla_Hg|_{\LTheta}/|\nabla_Eg|)d\sigma_E$,
  \[
    \Psi_f^*d\sigma_H
    =\frac{fu'}{\sqrt{1+f'^2}}\,
       f^{2n-1}\sqrt{1+f'^2}\,dt\,d\omega
    =f^{2n}u'\,dt\,d\omega=f^{2n}\,du\,d\omega.
  \]
  }
\end{proof}

\subsection{The Pansu sphere}
\label{subsec:pansu-sphere}
Let $h_{\mathrm P}(r)=\tfrac12(\arccos r+r\sqrt{1-r^2})$ for $0\le r\le1$ and
\[
  \Omega=\{(z,t):\abs z<1,\ \abs t<h_{\mathrm P}(\abs z)\},\qquad
  \PSphere=\partial\Omega.
\]
The outward-oriented $\PSphere$ has poles $(0,\pm\pi/4)$;
removing them gives the regular part $\PSphere^*$. The substitution
$|z|=\cos u$ combines the two caps in a single chart.  With
$I=(-\pi/2,\pi/2)$, define
\[
  \Psi(u,\zeta)=\bigl(\cos u\,e^{-iu}\zeta,
  \frac12(u+\sin u\cos u)\bigr).
\]
The phase $e^{-iu}$ makes the $u$-curves horizontal: for fixed $\zeta$
and $z=r e^{i\theta}\zeta$, the contact condition is
$t'+r^2\theta'=0$. Here $t'=r^2=\cos^2u$, so $\theta'=-1$;
choosing $\theta(0)=0$ gives the displayed phase.

For $u\ge0$ the height is $h_{\mathrm P}(\cos u)$; for $u\le0$ it is
$-h_{\mathrm P}(\cos u)$. Since $dt/du=\cos^2u>0$ on $I$, this is
a smooth diffeomorphism onto the regular part, including the equator.
Each endpoint collapses to one pole.

\begin{corollary}
  \label{prop:geometry}
  One has
  \begin{equation}
    \label{eq:pansu-pullback}
    \Psi^*\Theta=\cos^2u\,\vartheta,\qquad
    \left.\Psi^*\LTheta\right|_{\RR\partial_u\oplus\CRD}
    =du^2+\cos^2u\left.\gSphere\right|_{\CRD},\quad
    \Psi^*d\sigma_H=\cos^{2n}u\,du\,d\omega.
  \end{equation}
  Moreover, $d\Psi(\CRD)=\xi'$ and $d\Psi(iY)=J'd\Psi(Y)$ for
  $Y\in\CRD$.
  The frame and shape data on $\PSphere^*$ are
  \begin{equation}
    \label{eq:pansu-frame-shape}
    \begin{gathered}
      \alpha=\tan u,\qquad e_n=-d\Psi(\partial_u),\qquad
      e_n\alpha=-\sec^2u,\\
      \Shape(e_n)=2e_n,\qquad \left.\Shape\right|_{\xi'}=\Id,
      \qquad |\Shape|^2=2n+2.
    \end{gathered}
  \end{equation}
  In particular, the horizontal mean curvature is $H_{\PSphere}=2n$.
\end{corollary}

\begin{proof}
  Apply Theorem~\ref{thm:cr-isomorphism} with $t_*=0$ and
  \[
    t_{\mathrm P}(v)=\frac12(v+\sin v\cos v),\qquad
    f(t_{\mathrm P}(v))=\cos v.
  \]
  Since $t_{\mathrm P}'(v)=\cos^2v>0$, this defines a positive smooth
  $f$ on $(-\pi/4,\pi/4)$. The chain rule gives
  \[
    \begin{aligned}
      f'(t_{\mathrm P}(v))
      &=\frac{d(\cos v)/dv}{t_{\mathrm P}'(v)}
       =-\frac{\sin v}{\cos^2v},\\
      (U\circ t_{\mathrm P})'(v)
      &=f(t_{\mathrm P}(v))^{-2}t_{\mathrm P}'(v)
       =\frac{\cos^2v}{\cos^2v}=1,\\
      (u\circ t_{\mathrm P})'(v)
      &=\sqrt{f'(t_{\mathrm P}(v))^2+f(t_{\mathrm P}(v))^{-2}}\,
        t_{\mathrm P}'(v)\\
      &=\sqrt{\frac{\sin^2v}{\cos^4v}+\frac1{\cos^2v}}\,\cos^2v=1.
    \end{aligned}
  \]
  The normalization at $v=0$ gives $U(t_{\mathrm P}(v))=u(t_{\mathrm P}(v))=v$.
  Thus $\Psi_f(t_{\mathrm P}(u),\zeta)=\Psi(u,\zeta)$, and the theorem
  yields \eqref{eq:pansu-pullback}, the CR identities, and
  $e_n=-d\Psi(\partial_u)$.
  In the parallel horizontal frame, $d\Psi(\partial_u)=-ie^{-2iu}\zeta$,
  so $e_{2n}=-e^{-2iu}\zeta$.
  The vector $d\Psi(\TSphere)$ has horizontal part
  $i\cos u\,e^{-iu}\zeta$ and contact component $\cos^2u$;
  the tangency relation becomes
  \[
    0=-\ip{e_{2n}}{d\Psi(\TSphere)}+\alpha\cos^2u
      =\cos u\,\Re(i e^{iu})+\alpha\cos^2u
      =-\sin u\cos u+\alpha\cos^2u.
  \]
  Since $\cos u>0$, this gives $\alpha=\tan u$.

  For a $\gSphere$-orthonormal frame $\varepsilon_a$ of $\CRD$, set
  $e_a=(\cos u)^{-1}d\Psi(\varepsilon_a)$, $1\le a\le2n-1$, $a\ne n$.
  {Differentiating $e_{2n}$ in the fixed horizontal frame yields}
  \[
    {-\nabla^\xi_{e_n}e_{2n}=2e_n,\qquad
    -\nabla^\xi_{e_a}e_{2n}=\frac1{\cos u}e^{-2iu}\varepsilon_a
    =e_a-\tan u\,J(e_a)\quad(a\ne n).}
  \]
  Adding $\alpha J'$ gives
  \[
    \Shape(e_n)=2e_n,\qquad \Shape(e_a)=e_a\quad(a\ne n),
    \qquad |\Shape|^2=2n+2.
  \]
  Thus $H_{\PSphere}=2+(2n-2)=2n$, and
  $e_n\alpha=-\partial_u\tan u=-\sec^2u$.
\end{proof}

For $R>0$, $\delta_R(z,t)=(Rz,R^2t)$ gives the radius-$R$ sphere.
Perimeter and volume scale as $R^{2n+1}$ and $R^{2n+2}$, so
\[
  (\PH-\tfrac{2n}{R}V)(\delta_R\Sigma)
  =R^{2n+1}(\PH-2nV)(\Sigma).
\]
Fixed dilations preserve boundary regularity, so stability and the
unrestricted Hessian index are invariant under these dilations.

\subsection{Regular second variation}
\label{subsec:regular-variation}
Proposition~\ref{gen:prop:regularvariation} identifies the stationary
local Hessian of $\PH-2nV$ on compact regular supports with
$Q_{\PSphere}$, using $H_{\PSphere}=2n$. Its Jacobi expression is
$L=-\Delta_b^{\PSphere}-q_{\PSphere}$; the Friedrichs realization
is constructed in Section~\ref{subsec:form-closure}.

For the Pansu sphere, the frame computed in Corollary~\ref{prop:geometry}
gives $\alpha=\tan u$, $e_n\alpha=-\cos^{-2}u$, and
$|\Shape|^2=2n+2$.  Thus \eqref{eq:potential-symmetric-shape} becomes
\begin{equation}
  \label{eq:pansu-jacobi-potential}
  q_{\PSphere}=2n+2-\frac4{\cos^2u}
  +2(n+1)\tan^2u=\frac{2n-2}{\cos^2u}.
\end{equation}
Thus $L=-\Delta_b^{\PSphere}-(2n-2)\cos^{-2}u$ on
$C_c^\infty(\PSphere^*)$. Pulling its form back to the coordinates $\Psi$
separates the radial and angular terms.

The metric and measure in \eqref{eq:pansu-pullback} give, for every
$h\in C_c^\infty(\PSphere^*)$ with $\psi=h\circ\Psi$,
\[
  \abs{\nabla_b^{\PSphere}h}^2\circ\Psi=\abs{\partial_u\psi}^2+
  \cos^{-2}u\abs{\gradSphere\psi}^2,
\]
\begin{equation}
  \label{eq:separated-laplacian}
  (\Delta_b^{\PSphere}
  h)\circ\Psi=\cos^{-2n}u\,\partial_u(\cos^{2n}u\,\partial_u\psi)+\cos^{-2}u\DeltaSphere\psi.
\end{equation}
Substituting the gradient identity above, the perimeter density, and
\eqref{eq:pansu-jacobi-potential} into \eqref{gen:eq:regular-hessian}
and polarizing gives the symmetric bilinear form on
$C_c^\infty(I\times S^{2n-1};\RR)$ given by
\begin{equation}
  \label{eq:jacobi-form}
  \begin{aligned}
    Q(\varphi,\psi)=\int_I\!\int_{S^{2n-1}}
    &\left[\cos^{2n}u\,\partial_u\varphi\,\partial_u\psi\right.\\
    &\left.\quad+\cos^{2n-2}u
    \{\gSphere(\gradSphere\varphi,\gradSphere\psi)-(2n-2)\varphi\psi\}\right]
    \dd\omega\dd u.
  \end{aligned}
\end{equation}

\begin{remark}[{Comparison with Montefalcone's eigenvalue problem}]
  \label{rem:montefalcone-comparison}
  {On the unit Pansu sphere $\PSphere$, Montefalcone's radial variable is
  $\rho=r=\cos u$,}
  so the stability potential in~\cite[Equation~(18)]{MontefalconeStability} is
  $q_{\PSphere}=(2n-2)\sec^2u$, as in \eqref{eq:jacobi-form}.
  The spectral problem in~\cite[Theorem~2.10]{MontefalconeSpectrum} concerns
  $-\Delta_b^{\PSphere}$, with radial eigenvalues $k(k+2n)$ for
  $k\in\NN_0$; the Jacobi operator here is
  $L=-\Delta_b^{\PSphere}-q_{\PSphere}$.
  A weighted eigenvalue problem $L\varphi=\lambda\sec^2u\,\varphi$
  is different again: its operator is $\cos^2u\,L$ in
  $L^2(\sec^2u\,d\sigma_H)$, whereas the spectrum below refers
  to $L$ in $\mathcal X=L^2(d\sigma_H)$.
\end{remark}

\section{Real energy domain and Hopf harmonics}
\label{sec:closedform}
{The two characteristic poles correspond to the endpoints of $I$, where
the perimeter density vanishes and the Jacobi potential is singular.
{The change of weight $y=\cos^n u\,\varphi$ removes the
perimeter weight. The transformed finite-energy domain consists of
functions $y$ for which $y,\partial_u y\in L^2(I\times\Sphere)$,
both endpoint traces vanish, and
$\sec u\,\gradSphere y\in L^2(I\times\Sphere)$.} Endpoint Hardy
inequalities control the inverse-square terms, and angular truncation
followed by radial approximation shows that functions supported away
from the poles form a core.
}

\subsection{The energy space}
\label{subsec:form-domain}
Via pullback by $\Psi$, we identify $\mathcal X$ with the real weighted space
\[
  \left\{\varphi\colon I\times\Sphere\to\RR\text{ measurable}:
    \int_I\!\int_{\Sphere}\cos^{2n}u\,\varphi^2\,d\omega\,du<\infty\right\}.
\]
This identification is an isometry since
$\Psi^*d\sigma_H=\cos^{2n}u\,du\,d\omega$ and the poles have measure zero.
Throughout, $C_c^k(M)$ denotes the real-valued $C^k$ functions with
compact support in $M$, for $k\in\NN_0\cup\{\infty\}$.

\Needspace{10\baselineskip}
\begin{definition}[Energy space]
  \label{def:energy}
  The real energy space is
  \[
    \cE_n=\left\{\varphi\in\mathcal X:
      \norm\varphi_{\cE_n}^2<\infty\right\},
  \]
  where
  \begin{equation}
    \label{eq:energy-norm}
    \norm\varphi_{\cE_n}^2=
    \int_I\!\int_{S^{2n-1}}\left[
    \cos^{2n}u\abs{\partial_u\varphi}^2+
    \cos^{2n-2}u(\abs{\gradSphere\varphi}^2+\varphi^2)
    \right]\dd\omega\dd u.
  \end{equation}
  Derivatives are distributional, and $\gradSphere$ is the closed weak
  $\CRD$-gradient. The form $Q$ on $\cE_n$ is given by
  \eqref{eq:jacobi-form}, and $Q[\varphi]\coloneqq Q(\varphi,\varphi)$.
\end{definition}

Let $L$ be the Friedrichs operator associated with $Q$, constructed in
Section~\ref{subsec:form-closure}. The real angular Hilbert space is
\[
  \cH=L^2(\Sphere,d\omega).
\]
For $(d,m)\in\mathcal I$, let $\cH_{d,m}$ be the restrictions to
$\Sphere$ of real homogeneous harmonic polynomials $P$ of degree $d$
on $\RR^{2n}$ for which
\[
  -\TSphere^2(P|_{\Sphere})=m^2(P|_{\Sphere}).
\]
Write $M_{d,m}=\dim_{\RR}\cH_{d,m}$. The parameters
$\Lambda_{d,m}$ and $\varrho_{d,m}$ are defined in
Section~\ref{subsec:intro-spectrum}.

\subsection{Real spherical harmonics and Hopf frequencies}
{We use the complex spherical-harmonic decomposition
\cite{Folland} and the dimension and eigenvalue formulas in
\cite[Corollary~2.6 and Section~4]{CowlingKlimaSikora}, recast in
real Hopf-frequency form.}
The space $\cH_{0,0}$ consists of the constants, and
$\cH_{1,1}=\{\zeta\mapsto\langle v,\zeta\rangle_{\RR^{2n}}:
 v\in\RR^{2n}\}$ consists of the real coordinate functions.
{In the dimension formula below, a binomial coefficient
with lower entry $-1$ is zero.}
\Needspace{12\baselineskip}

\begin{proposition}
  \label{thm:angular}
  The real angular space has the orthogonal decomposition
  \[
    \cH=\widehat\bigoplus_{(d,m)\in\mathcal I}\cH_{d,m},\qquad
    -\DeltaSphere|_{\cH_{d,m}}=\Lambda_{d,m},
    \quad \Lambda_{d,m}=d(d+2n-2)-m^2.
  \]
  Put $a=(d+m)/2$, $b=(d-m)/2$, and let $\delta_{m0}$ be $1$ for
  $m=0$ and $0$ otherwise. Then
  \[
    M_{d,m}=(2-\delta_{m0})\left[
      \binom{n+a-1}{a}\binom{n+b-1}{b}
      -\binom{n+a-2}{a-1}\binom{n+b-2}{b-1}\right].
  \]
  The least positive eigenvalue of $-\DeltaSphere$ is $2n-2$;
  its eigenspace is $\cH_{1,1}$ and has real dimension $2n$.
\end{proposition}
{The proof is given in Appendix~\ref{app:angular-decomposition}.}

Fix a real smooth orthonormal eigenbasis $(Y_j)_{j\ge0}$ of
$-\DeltaSphere$, with eigenvalues $\Lambda_j$. The angular form domain is
\[
  \operatorname{Dom}((-\DeltaSphere)^{1/2})
  =\left\{h\in\cH:
    \int_{\Sphere}|\gradSphere h|^2\,d\omega<\infty\right\},
\]
with the gradient understood weakly. For every $h$ in this space,
the finite spectral sums converge to $h$ in the form norm of
$-\DeltaSphere$. In particular,
\[
  \int_{S^{2n-1}}\abs{\gradSphere h}^2\dd\omega
  =\sum_j\Lambda_j\ip{Y_j}{h}_{\cH}^2.
\]

\subsection{The Friedrichs form domain}
\label{subsec:form-closure}
The inclusion $C_c^\infty(I\times\Sphere)\subset\cE_n$ gives density in
$\mathcal X$, and $\|\varphi\|_{\mathcal X}\le\|\varphi\|_{\cE_n}$.
The integral formula and Cauchy--Schwarz show that $Q$ is symmetric and
$|Q(\varphi,\psi)|\le C_n\|\varphi\|_{\cE_n}\|\psi\|_{\cE_n}$.
We use the representation theorem for closed semibounded forms
\cite[Chap.~VI, Theorems~2.1 and~2.6]{Kato}.

{Define $U\varphi=\cos^n u\,\varphi$. Then $U$ maps
$\mathcal X$ isometrically onto}
\[
  \widetilde{\mathcal X}
  =\left\{y\colon I\times\Sphere\to\RR\text{ measurable}:
    \int_I\!\int_{\Sphere}|y|^2\,d\omega\,du<\infty\right\}
  =L^2(I;\cH).
\]
The radial Sobolev spaces are
\[
  \begin{aligned}
    H^1(I;\cH)
      &=\{y\in L^2(I;\cH):\partial_u y\in L^2(I;\cH)\},\\
    H_0^1(I;\cH)
      &=\{y\in H^1(I;\cH):y(-\pi/2)=y(\pi/2)=0\}.
  \end{aligned}
\]
The derivative is distributional, and the endpoint values are traces.
The scalar spaces $H^1(I)$ and $H_0^1(I)$ are obtained by replacing
$\cH$ with $\RR$. Define
\[
  \widetilde{\cE}_n
  =\left\{y\in H_0^1(I;\cH):
    \int_I\!\int_{\Sphere}\sec^2u\,|\gradSphere y|^2\,d\omega\,du
    <\infty\right\},
\]
with norm
\[
  \|y\|_{\widetilde{\cE}_n}^2
  =\|\partial_u y\|_{\widetilde{\mathcal X}}^2
  +\|\sec u\,\gradSphere y\|_{\widetilde{\mathcal X}}^2.
\]
Horizontal vector norms use $\gSphere$. The zero trace applies to
$y=U\varphi$, allowing the original speed $\varphi$ to be singular
at a pole. For example, central translation has speed $\tan u$, whose
conjugate $\sin u\cos^{n-1}u$ belongs to $H_0^1(I)$ for $n\ge2$.
Thus the endpoint condition is compatible with moving the poles;
Section~\ref{sec:singular} establishes finite energy for all $C^1$
surface fields.

\begin{proposition}[Friedrichs form domain]
  \label{lem:hardydensity}
  \label{prop:closedform}
  The space $\widetilde{\cE}_n$ is Hilbert in its displayed norm, and
  \[
    U\colon \cE_n\xrightarrow{\;\cong\;}\widetilde{\cE}_n,\qquad
    C_n^{-1}\norm\varphi_{\cE_n}
    \le\norm{U\varphi}_{\widetilde{\cE}_n}
    \le C_n\norm\varphi_{\cE_n}
    \quad(\varphi\in\cE_n),
  \]
  where $C_n\ge1$ depends only on $n$.
  For $y=U\varphi$ and $\varphi\in C_c^\infty(I\times S^{2n-1})$,
  \begin{equation}
    \label{eq:transformed-global-form}
    \begin{aligned}
      Q[\varphi]+n^2\norm y_{\widetilde{\mathcal X}}^2
      &=\norm{\partial_u y}_{\widetilde{\mathcal X}}^2+\norm{\sec u\,\gradSphere
      y}_{\widetilde{\mathcal X}}^2\\
      &\quad +(n-1)(n-2)\norm{\sec u\,y}_{\widetilde{\mathcal X}}^2.
    \end{aligned}
  \end{equation}

  The space $C_c^\infty(I\times S^{2n-1};\RR)$ is dense in $\cE_n$
  and is a form core for the closed, symmetric bilinear form $Q$.
  With $\langle\varphi,\psi\rangle_{\mathcal X}
  =\int\varphi\psi\,d\sigma_H$, its associated
  self-adjoint operator $L$ is the Friedrichs realization, with
  \[
    \operatorname{Dom}(L)
    =\left\{\psi\in\cE_n:
      \begin{array}{l}
        \text{there exists }g\in\mathcal X\text{ such that}\\
        Q(\varphi,\psi)=\ip{\varphi}{g}_{\mathcal X}
        \text{ for every }\varphi\in\cE_n
      \end{array}\right\},
    \qquad L\psi=g.
  \]
\end{proposition}

\begin{proof}
  For compactly supported $\varphi$, set $y=\cos^nu\,\varphi$.
  The product rule gives
  \[
    \partial_u\varphi=\cos^{-n}u\bigl(\partial_uy+n\tan u\,y\bigr),
    \qquad \gradSphere\varphi=\cos^{-n}u\,\gradSphere y.
  \]
  All norms and integrals in the next calculation use $du\,d\omega$.
  Since $y$ is compactly supported, its radial cross term is
  \[
    \begin{aligned}
      2n\int\tan u\,y\,\partial_uy
      &=n\int\tan u\,\partial_u|y|^2\\
      &=-n\int\sec^2u\,|y|^2=-n\|\sec u\,y\|_2^2.
    \end{aligned}
  \]
  Substituting in \eqref{eq:jacobi-form}, we obtain
  \[
    \begin{aligned}
      Q[\varphi]
      &=\|\partial_uy+n\tan u\,y\|_2^2
        +\|\sec u\,\gradSphere y\|_2^2-(2n-2)\|\sec u\,y\|_2^2\\
      &=\|\partial_uy\|_2^2+\|\sec u\,\gradSphere y\|_2^2
        +n^2\|\tan u\,y\|_2^2-(3n-2)\|\sec u\,y\|_2^2\\
      &=\|\partial_uy\|_2^2+\|\sec u\,\gradSphere y\|_2^2
        +(n-1)(n-2)\|\sec u\,y\|_2^2-n^2\|y\|_2^2.
    \end{aligned}
  \]
  Here $\tan^2u=\sec^2u-1$ and
  $n^2-(3n-2)=(n-1)(n-2)$ give \eqref{eq:transformed-global-form}.

  Poincar\'e's inequality makes a Cauchy sequence in $\widetilde{\cE}_n$
  converge in $H_0^1(I;\cH)$; closedness of the weak CR gradient identifies
  the $L^2$ limit of its weighted angular gradients as
  $\sec u\,\gradSphere y$. Thus $\widetilde{\cE}_n$ is complete.
  Put $d(u)=\pi/2-|u|$. Integration by parts on each half of $I$
  and Cauchy--Schwarz give the scalar endpoint Hardy inequality
  \[
    \int_I\frac{|g(u)|^2}{d(u)^2}\,du
    \le4\int_I|g'(u)|^2\,du,
    \qquad g\in C_c^\infty(I).
  \]
  Apply this inequality to the coordinates of $y$ in an orthonormal
  basis of $\cH$, sum by Parseval and pass to the $H_0^1$ closure.
  The resulting Hilbert-valued inequality and
  $\cos u\ge(2/\pi)(\pi/2-|u|)$ give
  \begin{equation}
    \label{eq:h0-endpoint-hardy}
    \|\sec u\,y\|_2^2\le C\|\partial_u y\|_2^2,
    \qquad y\in H_0^1(I;\cH).
  \end{equation}
  For $y=U\varphi$, direct substitution gives
  \[
    \|\varphi\|_{\cE_n}^2
    =\|\partial_u y+n\tan u\,y\|_2^2
    +\|\sec u\,\gradSphere y\|_2^2+\|\sec u\,y\|_2^2.
  \]
  Hardy's inequality gives $\|U^{-1}y\|_{\cE_n}\le C_n\|y\|_{\widetilde{\cE}_n}$.
  Conversely, the displayed norm identity controls
  $\partial_uy+n\tan u\,y$, $\sec u\,y$ and
  $\sec u\,\gradSphere y$. In particular,
  \[
    \begin{aligned}
      \|\partial_uy\|_2^2
      &\le2\|\partial_uy+n\tan u\,y\|_2^2
        +2n^2\|\tan u\,y\|_2^2\\
      &\le2\|\partial_uy+n\tan u\,y\|_2^2
        +2n^2\|\sec u\,y\|_2^2
       \le C_n\|\varphi\|_{\cE_n}^2.
    \end{aligned}
  \]
  Hence $y\in H^1(I;\cH)\subset C(\overline I;\cH)$.
  A nonzero endpoint trace would contradict $\sec u\,y\in L^2$.
  Thus $y\in H_0^1(I;\cH)$ and the reverse norm bound follows.

  Let $y=U\varphi\in\widetilde{\cE}_n$. Finite angular spectral projections
  $P_Ny$ converge to $y$ in $\widetilde{\cE}_n$ by Parseval. Each $P_Ny$
  has finitely many {radial components} $y_j\in H_0^1(I)$.
  Choose $h_j\in C_c^\infty(I)$ converging to $y_j$ in $H_0^1$.
  Hardy's inequality gives
  $\|\sec u\,(h_j-y_j)\|_2\le C\|(h_j-y_j)'\|_2$,
  so the finite weighted angular term also converges.
  A diagonal choice and the norm equivalence just proved give density in
  $\cE_n$.

  Density extends \eqref{eq:transformed-global-form} to $\cE_n$.
  Hardy's inequality and the same norm equivalence then give, including when $n=2$,
  \begin{equation}
    \label{eq:shifted-norm}
    C_n^{-1}\|\varphi\|_{\cE_n}^2\le
    Q[\varphi]+n^2\|\varphi\|_{\mathcal X}^2
    \le C_n\|\varphi\|_{\cE_n}^2.
  \end{equation}
  Completeness proves closedness. The representation theorem and the
  form core identify $L$ as the Friedrichs realization.
\end{proof}

\section{Global second variation and dilation pairing}
\label{sec:singular}

\subsection{Pole estimates and energy cutoffs}
\label{sec:pansu-hessian}
Fix $0<r_0<1$. In the two pole charts $B_{r_0}\subset\RR^{2n}$,
use $t=\pm h_{\mathrm P}(r)$, $r=|z|$, with the profile from
Section~\ref{subsec:pansu-sphere}; thus $h_{\mathrm P}'(r)=-r^2/\sqrt{1-r^2}$.
Let $G_\pm$ be the horizontal cofactors for the outward orientation,
so the perimeter density in each chart is $|G_\pm|\,dz$.
In the next proposition, the $C^1$ norm is taken on the fixed surface
charts, and $C$ depends only on $n,r_0$ and these charts, not on the
field or the cap radius.

\begin{proposition}
  \label{prop:poleexpansion}
  {At both Pansu poles, $D_zG_\pm(0)$ is invertible.}
  For $n\ge2$, the normalized normal speed $\varphi_W$ of every $C^1$
  surface field $W$ belongs to $\cE_n$ and satisfies, in either cap,
  \begin{equation}
    \label{eq:pole-energy-bound}
    \int_{r<\varepsilon}
    \left(|\nabla_b \varphi_W|^2+q_{\PSphere}\varphi_W^2\right)d\sigma_H
    \le C\varepsilon^{2n-3}\|W\|_{C^1}^2
    \qquad(0<\varepsilon<r_0).
  \end{equation}
\end{proposition}
\begin{proof}
  {Since $h_{\mathrm P}'(r)=-r^2/\sqrt{1-r^2}$, radial
  differentiation gives the outward horizontal cofactors}
  \begin{equation}
    \label{eq:exact-pole-cofactors}
    {G_\pm=
    \left(\frac r{\sqrt{1-r^2}}\Id\mp J\right)z.}
  \end{equation}
  Using $\ip z{Jz}=0$, we obtain
  \begin{equation}
    \label{eq:exact-pole-cofactor-norm}
    |G_\pm(z)|=\frac r{\sqrt{1-r^2}},\qquad D_zG_\pm(0)=\mp J.
  \end{equation}
  Taylor expansion and \eqref{eq:pansu-jacobi-potential} give
  \begin{equation}
    \label{eq:canonical-pole-orders}
    \begin{gathered}
      \pm h_{\mathrm P}(r)=\pm\frac\pi4\mp\frac13r^3+O(r^5),
      \qquad |G_\pm|=r+O(r^3),\\
      d\sigma_H=(r^{2n}+O(r^{2n+2}))\,dr\,d\omega,
      \qquad q_{\PSphere}=\frac{2n-2}{r^2}.
    \end{gathered}
  \end{equation}
  {Differentiating the radial expansion twice shows that its Hessian is
  Lipschitz at $z=0$, so the caps are $C^{2,1}$ there. The nonzero
  $|z|^3$ term prevents $C^3$ regularity.}
  {At each pole, $D_zG_\pm(0)=\mp J$ is invertible, so
  Proposition~\ref{gen:prop:characteristic-energy} applies. Since
  $q_{\PSphere}=(2n-2)/r^2\ge0$, its cap estimate
  \eqref{gen:eq:characteristic-energy-cap} gives
  \eqref{eq:pole-energy-bound}. Its full-rank bounds
  \eqref{gen:eq:characteristic-energy-pointwise} also give
  $|\varphi_W|\le Cr^{-1}\|W\|_{C^1}$ and
  $|\nabla_b\varphi_W|\le Cr^{-2}\|W\|_{C^1}$, as used in the
  cutoff estimate below. Away from the poles, $\varphi_W$ is $C^1$.
  Finally, $q_{\PSphere}/(2n-2)=\cos^{-2}u$, and the metric and
  measure in \eqref{eq:pansu-pullback} identify the cap estimate with
  the energy terms in \eqref{eq:energy-norm}. Thus
  $\varphi_W\in\cE_n$.}
\end{proof}

The same pole orders give the energy cutoff estimate in a general
power form. If $|\phi|\le Mr^\ell$ and
$|\nabla_b \phi|\le Mr^{\ell-1}$ in a pole cap, with
$2n+2\ell-1>0$, then, for $0<2\varepsilon<r_0$, a radial cutoff $\chi_\varepsilon$ equal
to
$0$ for $r\le\varepsilon$ and $1$ for $r\ge2\varepsilon$, with
$|D\chi_\varepsilon|\le C/\varepsilon$, satisfies
\begin{equation}
  \label{eq:pole-power-estimate}
  \int_{r<\varepsilon}(|\nabla_b \phi|^2+q_{\PSphere}|\phi|^2)\,d\sigma_H
  +\|(1-\chi_\varepsilon)\phi\|_{\cE_n}^2
  \le CM^2\varepsilon^{2n+2\ell-1}.
\end{equation}
Indeed,
$\nabla_b((1-\chi_\varepsilon)\phi)
=(1-\chi_\varepsilon)\nabla_b\phi-\phi\nabla_b\chi_\varepsilon$.
The squared triangle inequality and the pole measure therefore give
\[
  \begin{aligned}
    \|(1-\chi_\varepsilon)\phi\|_{\cE_n}^2
    &\le CM^2\left(\int_0^{2\varepsilon}r^{2n+2\ell-2}\,dr
       +\varepsilon^{-2}\int_\varepsilon^{2\varepsilon}
            r^{2n+2\ell}\,dr\right)\\
    &\le CM^2\left(\frac{(2\varepsilon)^{2n+2\ell-1}}{2n+2\ell-1}
       +\varepsilon^{-2}
        \frac{(2\varepsilon)^{2n+2\ell+1}}{2n+2\ell+1}\right)\\
    &\le CM^2\varepsilon^{2n+2\ell-1}.
  \end{aligned}
\]
The uncut cap integral is bounded by the first radial integral above. In the energy norm, $(1-\chi_\varepsilon)\phi$ is
extended by zero outside the pole chart.

\begin{proof}[Proof of \cref{cor:intro-hessian}]
  {The sphere is a compact $C^2$ boundary. Its characteristic
  set consists of the two poles, and
  \eqref{eq:exact-pole-cofactor-norm} gives
  $\operatorname{rank}D_zG_\pm(0)=2n\ge4$.
  Corollary~\ref{prop:geometry} gives $H_{\PSphere}=2n$ on the regular set.
  Theorem~\ref{thm:c2-global-hessian}, with $h=2n$, gives
  stationarity and the continuous Hessian.
  Proposition~\ref{gen:prop:rank-global-jacobi} gives its global pairing with
  $Q_{\PSphere}$ through \eqref{eq:abstract-global-hessian} and the
  one-parameter second-derivative identity.
  Proposition~\ref{prop:poleexpansion} gives $\varphi_W\in\cE_n$.
  By \eqref{eq:jacobi-form} and \eqref{eq:energy-norm}, the general
  energy in \eqref{gen:eq:abstract-energy-cutoff} controls the
  $\cE_n$ norm at the poles. Hence the pole cutoffs converge in
  $\cE_n$, and the improper pairing $Q_{\PSphere}$ equals the closed
  form $Q$ on these speeds. This proves the corollary.}
\end{proof}

\subsection{The dilation pairing and volume constraint}
{For a surface field $W$, the volume constraint is
$\delta V[W]=\int_\PSphere\varphi_W\,d\sigma_H=0$.
We express it through the Jacobi pairing with the dilation field.}
Let $D=\sum_j(x_j\partial_{x_j}+y_j\partial_{y_j})
+2t\partial_t$, the generator of $\delta_{e^\tau}$. On the unit Pansu sphere,
\begin{equation}
  \label{eq:dilation-support-speed}
  \varphi_D=\cos^2u+(u+\sin u\cos u)\tan u
  =1+u\tan u.
\end{equation}
The bounds $\varphi_D=O(\cos^{-1}u)$ and
$\partial_u\varphi_D=O(\cos^{-2}u)$ give $\varphi_D\in\cE_n$ for $n\ge2$.
For $f=\varphi_D=1+u\tan u$, the derivatives and radial operator give
\[
  f'=\tan u+u\sec^2u,\qquad
  f''=2\sec^2u+2u\sec^2u\tan u,
\]
\[
  \begin{aligned}
    Lf&=-f''+2n\tan u\,f'-(2n-2)\sec^2u\,f\\
      &=-2\sec^2u-2u\sec^2u\tan u
        +2n\tan^2u+2nu\tan u\sec^2u\\
      &\quad-(2n-2)\sec^2u-(2n-2)u\sec^2u\tan u\\
      &=-2n(\sec^2u-\tan^2u)=-2n.
  \end{aligned}
\]
Integration by parts against compact regular tests therefore gives
$Q(\varphi_D,\varphi)=\int(L\varphi_D)\varphi\,d\sigma_H$.
Density and continuity of the closed form extend this to
\begin{equation}
  \label{eq:model-dilation-pairing}
  Q(\varphi_D,\varphi)=-2n\int_\PSphere\varphi\,d\sigma_H
  \quad(\varphi\in\cE_n),\qquad Q[\varphi_D]<0.
\end{equation}
Thus the volume constraint is precisely $Q$-orthogonality to this negative direction.

\section{Complete real spectrum, translation kernel, and stability}
\label{sec:spectrum}
\label{subsec:angular-decomposition}
\label{subsec:kernel-geometry}
We solve the separated equations in the finite-energy domain.
{For $\varrho>0$, let $C_k^{(\varrho)}$ be the degree-$k$ Gegenbauer
polynomial with the standard normalization
$C_k^{(\varrho)}(1)=(2\varrho)_k/k!$.}

\begin{proposition}[Global eigenbasis]
  \label{prop:global-eigenbasis}
  Choose a real orthonormal basis $Y_{d,m,\ell}$ of $\cH_{d,m}$,
  $1\le\ell\le M_{d,m}$, for each $(d,m)\in\mathcal I$. The functions
  \begin{equation}
    \label{eq:global-eigenfunctions}
    \varphi_{d,m,k,\ell}
    =\cos^{\varrho_{d,m}-n}u\,
    C_k^{(\varrho_{d,m})}(\sin u)Y_{d,m,\ell},
    \qquad (d,m)\in\mathcal I,\ k\in\NN_0,
  \end{equation}
  form a complete orthogonal basis of $\mathcal X$ and belong to
  $\operatorname{Dom}(L)$, with eigenvalues
  $\mu_{d,m,k}=(\varrho_{d,m}+k)^2-n^2$.
\end{proposition}
{The reduction to the Gegenbauer equation is given below;
the domain and completeness arguments are in
Appendix~\ref{app:global-eigenbasis}.}

Let $Y$ be a smooth angular eigenfunction with
$-\DeltaSphere Y=\Lambda Y$ and $\|Y\|_{\cH}=1$.
For $\varphi=fY$, the energy norm is
\[
  \|fY\|_{\cE_n}^2
  =\int_I\left[\cos^{2n}u|f'|^2
  +(1+\Lambda)\cos^{2n-2}u|f|^2\right]\dd u.
\]
By Proposition~\ref{lem:hardydensity} and Hardy's inequality,
$fY\in\cE_n$ if and only if $y=\cos^n u\,f\in H_0^1(I)$.
The separated Jacobi expression acts on $f$ as
\[
  \begin{aligned}
    L_\Lambda f
    &=-\cos^{-2n}u\,(\cos^{2n}u\,f')'
      +(\Lambda-2n+2)\sec^2u\,f\\
    &=-f''+2n\tan u\,f'+(\Lambda-2n+2)\sec^2u\,f.
  \end{aligned}
\]
With $f=\cos^{-n}u\,y$, the derivatives are
\[
  \begin{aligned}
    f'&=\cos^{-n}u\,(y'+n\tan u\,y),\\
    f''&=\cos^{-n}u\bigl(y''+2n\tan u\,y'
                   +(n\sec^2u+n^2\tan^2u)y\bigr).
  \end{aligned}
\]
Substituting and cancelling $y'$ gives
\[
  \begin{aligned}
    \cos^nu\,L_\Lambda(\cos^{-n}u\,y)
    &=-y''+\bigl(n^2\tan^2u+(\Lambda-3n+2)\sec^2u\bigr)y\\
    &=-y''+\bigl((n-1)(n-2)+\Lambda\bigr)\sec^2u\,y-n^2y.
  \end{aligned}
\]
Equivalently, the conjugated expression is
\begin{equation}
  \label{eq:poschl-teller}
  -y''+\varrho(\Lambda)(\varrho(\Lambda)-1)\sec^2u\,y-n^2y,
  \qquad
  \varrho(\Lambda)(\varrho(\Lambda)-1)=(n-1)(n-2)+\Lambda.
\end{equation}
Here $\varrho(\Lambda)\ge1$. When $n=2$ and $\Lambda=0$, the
expression is $-y''-4y$ with zero endpoint traces.

Gegenbauer polynomials $C_k^{(\varrho)}$, $\varrho\ge1$, form an
orthogonal basis of \mbox{$L^2((-1,1),(1-x^2)^{\varrho-1/2}\dd x)$}
\cite[Theorem~3.1.5 and \S4.7]{Szego}. They satisfy
\[
  (1-x^2)(C_k^{(\varrho)})''-(2\varrho+1)x(C_k^{(\varrho)})'
  +k(k+2\varrho)C_k^{(\varrho)}=0.
\]
To see the reduction explicitly, put $\varrho=\varrho(\Lambda)$,
$x=\sin u$, $c=\cos u$ and $y=c^\varrho v(x)$; primes on $v$ denote
$x$-derivatives. Then
\[
  \begin{aligned}
    y'&=c^{\varrho-1}\bigl(-\varrho x v+c^2v'\bigr),\\
    y''&=c^{\varrho-2}\bigl(c^4v''-(2\varrho+1)xc^2v'
            +(\varrho(\varrho-1)x^2-\varrho c^2)v\bigr).
  \end{aligned}
\]
Thus the expression in \eqref{eq:poschl-teller} becomes
\[
  \begin{aligned}
    -y''+\varrho(\varrho-1)c^{-2}y-n^2y
    &=c^\varrho\bigl(-(1-x^2)v''+(2\varrho+1)xv'
           +(\varrho^2-n^2)v\bigr).
  \end{aligned}
\]
The eigenvalue equation is consequently
\[
  (1-x^2)v''-(2\varrho+1)xv'
       +(\mu+n^2-\varrho^2)v=0.
\]
For $v=C_k^{(\varrho)}$, comparison with its differential equation gives
\[
  \mu+n^2-\varrho^2=k(k+2\varrho),\qquad
  \mu=\varrho^2+2k\varrho+k^2-n^2=(\varrho+k)^2-n^2.
\]

\begin{proof}[{Proof of Theorem~\ref{thm:main-spectrum}}]
  By Proposition~\ref{prop:global-eigenbasis}, each eigenvalue
  $\mu$ has multiplicity
  \begin{equation}
    \label{eq:multiplicity}
    \sum_{\substack{(d,m)\in\mathcal I,\ k\in\NN_0\\
                    \mu_{d,m,k}=\mu}}M_{d,m}.
  \end{equation}
  To prove finiteness, fix $R\ge0$.
  If $\mu_{d,m,k}\le R$, then
  \[
    k\le\sqrt{R+n^2},\qquad \varrho_{d,m}\le\sqrt{R+n^2},
  \]
  \[
    2(n-1)d\le\Lambda_{d,m}
    =\varrho_{d,m}(\varrho_{d,m}-1)-(n-1)(n-2)
    \le R+n^2.
  \]
  Thus $\#\{(d,m,k):\mu_{d,m,k}\le R\}<\infty$.
  The eigenvalues tend to $+\infty$, proving \eqref{eq:main-spectrum}.
  On the normalized eigenbasis, $(L+n^2)^{-1}$ is diagonal with entries
  $(\mu_{d,m,k}+n^2)^{-1}\to0$. Its finite spectral truncations converge
  in operator norm, so the resolvent is compact.

  In the constant angular sector, $\varrho(0)=n-1$ gives eigenvalues
  $-(2n-1),0,2n+1$ at $k=0,1,2$, with radial factors $\sec u,\tan u$
  at $k=0,1$. Proposition~\ref{thm:angular} and monotonicity give
  $\varrho(2n-2)=n$ and $\varrho(\Lambda)>n$ for $\Lambda>2n-2$.
  Thus $\varrho(\Lambda)+k=n$ only at $(\Lambda,k)=(0,1),(2n-2,0)$.
  The latter has constant radial factor and angular space
  $\cH_{1,1}$ of real dimension $2n$.
  The constant sector at $k=0$ gives the sole negative eigenvalue,
  $-(2n-1)$, with multiplicity $1$. The eigenbasis therefore gives the
  optimal form bound
  \[
    Q[\varphi]\ge-(2n-1)\|\varphi\|_{\mathcal X}^2,
    \qquad\text{with equality exactly for }\varphi=c\sec u,
    \quad c\in\RR.
  \]

  The real eigenbasis gives index one and the kernel
  \[
    \begin{aligned}
      \cK_n=\ker L
      &=\{\varphi\in\operatorname{Dom}(L):L\varphi=0\}\\
      &=\{(u,\zeta)\mapsto c\tan u
         +\langle v,\zeta\rangle_{\RR^{2n}}:
        c\in\RR,\ v\in\RR^{2n}\},
      \qquad \dim_\RR\cK_n=2n+1.
    \end{aligned}
  \]

  For $v\in\RR^{2n}$, let $a\in\CC^n$ be its coordinate
  vector, and let $a_b$ be the left-invariant horizontal field
  {whose value at the identity is $a$}. Differentiating
  left translation in direction $a$
  {gives $K_a=a_b-2\Im\ip zaT$. Set $A=\ip\zeta a$.} Substituting
  $z=\cos u\,e^{-iu}\zeta$, $e_{2n}=-e^{-2iu}\zeta$, and
  $\alpha=\tan u$ into the normalized normal speed gives
  \[
    \begin{aligned}
      \frac{\ip{K_a+cT}{N}}{\abs{N_b}}
      &=\Re(e^{-2iu}A)+2\sin u\,\Im(e^{iu}\overline A)+c\tan u\\
      &=\cos(2u)\Re A+\sin(2u)\Im A
         +2\sin u\bigl(\sin u\Re A-\cos u\Im A\bigr)+c\tan u\\
      &=(\cos(2u)+2\sin^2u)\Re A+c\tan u\\
      &=\Re A+c\tan u
       =\langle v,\zeta\rangle_{\RR^{2n}}+c\tan u.
    \end{aligned}
  \]
  As $a$ and $c$ vary, these speeds exhaust $\cK_n$.
\end{proof}

\begin{proof}[Proof of \cref{cor:intro-stability}]
  {On the unit Pansu sphere $\PSphere$,
  Theorem~\ref{thm:main-spectrum}
  and Corollary~\ref{cor:intro-hessian} give
  $\operatorname{ind}\mathcal B_{2n}\le1$.}
  For the dilation field $D$, \eqref{eq:model-dilation-pairing} gives
  \[
    Q[\varphi_D]<0,\qquad
    Q(\varphi_D,\varphi)=-2n\int_\PSphere\varphi\,d\sigma_H.
  \]
  Thus $\operatorname{ind}\mathcal B_{2n}=1$. Put $f=\varphi_D$,
  $\ell(\varphi)=\int_\PSphere\varphi\,d\sigma_H$, and
  $M=\ker\ell\subset\cE_n$. The index-one property implies
  $Q|_M\ge0$, since a negative vector in $M$ is $Q$-orthogonal to
  $f$ and would give a negative-definite plane. If $\varphi\in M$
  and $Q[\varphi]=0$, then $Q[\varphi+t\psi]\ge0$ for all
  $\psi\in M$ and $t\in\RR$ forces $Q(\varphi,\psi)=0$.
  Also $Q(\varphi,f)=0$ and $\ell(f)>0$, so
  $\cE_n=M\oplus\spn\{f\}$ and $\varphi$ is $Q$-orthogonal to
  the entire form domain. The representation theorem gives
  $\varphi\in\operatorname{Dom}(L)$ and $L\varphi=0$.
  Conversely, kernel elements are translation speeds with zero mean.
  The global Hessian identity therefore gives exactly the stated
  zero directions for $C^1$ surface fields. Stationarity identifies
  this Hessian with the second perimeter derivative along
  volume-preserving variations. Translation and Heisenberg dilation
  give the conclusion for every Pansu sphere.
\end{proof}

\appendix
\section{\texorpdfstring{{Real angular decomposition}}{Real angular decomposition}}
\label{app:angular-decomposition}

{We recall the classical complex spherical-harmonic
decomposition~\cite{Folland} and the dimension and eigenvalue formulas
of \cite[Corollary~2.6 and Section~4]{CowlingKlimaSikora} in the real
Hopf-frequency notation used here. This makes explicit the pairing of
opposite nonzero frequencies and the count at zero frequency.}

{\begin{proof}[Proof of Proposition~\ref{thm:angular}]
  Real spherical harmonics of each homogeneous degree $d$ form a
  complete orthogonal decomposition of $\cH$. The Hopf circle action
  preserves degree and harmonicity. Its real Fourier decomposition on
  each degree space is orthogonal, and a degree-$d$ monomial
  $z^\alpha\bar z^\beta$ has frequency
  $|\alpha|-|\beta|$. Hence the possible absolute frequencies are
  exactly $0\le m\le d$ with $d-m$ even, giving the stated real
  decomposition.

  The real polynomial space of degree $d$ and absolute Hopf frequency
  $m$ has dimension
  $(2-\delta_{m0})\binom{n+a-1}{a}\binom{n+b-1}{b}$:
  positive frequency has cosine and sine components, while zero
  frequency is counted once. The Fischer decomposition into harmonic
  polynomials and $|z|^2$ times degree-$(d-2)$ polynomials preserves
  Hopf frequency. Subtracting the latter dimension, with indices
  $a-1,b-1$, gives $M_{d,m}$.

  On $\cH_{d,m}$, the round Laplacian is
  $-\Delta^{S^{2n-1}}=d(d+2n-2)$ and
  $-\TSphere^2=m^2$. The orthogonal splitting
  $T\Sphere=\CRD\oplus\RR\TSphere$ gives
  $\Delta^{S^{2n-1}}=\DeltaSphere+\TSphere^2$, hence
  $-\DeltaSphere=\Lambda_{d,m}$. Equivalently,
  $d=a+b$ and $m=a-b$ give
  $\Lambda_{d,m}=4ab+2(n-1)(a+b)$. Finally,
  $\Lambda_{d,m}\ge2(n-1)d\ge2n-2$ for $d\ge1$,
  with equality only for $(d,m)=(1,1)$.
\end{proof}}

\section{\texorpdfstring{{Proof of the global eigenbasis}}{Proof of the global eigenbasis}}
\label{app:global-eigenbasis}

{The radial input is the classical Gegenbauer basis
\cite[Theorem~3.1.5 and \S4.7]{Szego}; Proposition~\ref{thm:angular}
provides the real angular basis. We verify that the resulting separated
functions lie in $\operatorname{Dom}(L)$ and form a complete basis of
$\mathcal X$.}

{\begin{proof}[Proof of Proposition~\ref{prop:global-eigenbasis}]
  Fix $(d,m)\in\mathcal I$ and $k,\ell$, and put
  $\varrho=\varrho_{d,m}$.
  The function $y_k=\cos^\varrho u\,C_k^{(\varrho)}(\sin u)$
  vanishes to order at least $1$ at each endpoint, so
  $y_k\in H_0^1(I)$. Hardy's inequality also gives
  $\sec u\,y_k\in L^2(I)$; hence
  $\varphi_{d,m,k,\ell}\in\cE_n$.
  The Gegenbauer and angular equations give, by integration by parts,
  \[
    Q(\varphi_{d,m,k,\ell},\psi)
    =\mu_{d,m,k}\langle\varphi_{d,m,k,\ell},\psi\rangle_{\mathcal X},
    \qquad \psi\in C_c^\infty(I\times\Sphere).
  \]
  Proposition~\ref{prop:closedform} extends this identity to every
  $\psi\in\cE_n$. The representation theorem therefore puts
  $\varphi_{d,m,k,\ell}$ in $\operatorname{Dom}(L)$ with the stated eigenvalue.

  For each fixed $(d,m)$, the change of variables $x=\sin u$ gives the
  real isometry
  \[
    \begin{aligned}
      L^2((-1,1),(1-x^2)^{\varrho-1/2}\dd x)&\longrightarrow L^2(I),\\
      h&\longmapsto\cos^\varrho u\,h(\sin u).
    \end{aligned}
  \]
  Indeed $dx=\cos u\,du$ gives
  \[
    \int_I|\cos^\varrho u\,h(\sin u)|^2\,du
    =\int_{-1}^1(1-x^2)^{\varrho-1/2}|h(x)|^2\,dx.
  \]
  It sends the Gegenbauer basis to $(y_k)_{k\ge0}$, giving a complete
  orthogonal radial basis in each angular sector. The orthogonal direct
  sum in Proposition~\ref{thm:angular} then gives global completeness.
\end{proof}}

\Needspace{10\baselineskip}


\begin{thebibliography}{99}
  \linespread{1.0}
  \selectfont
  \setlength{\itemsep}{3pt}
  \setlength{\parsep}{0pt}
  \bibitem{Balogh}
  Z.~M. Balogh,
  \href{https://doi.org/10.1515/crll.2003.094}{\emph{Size of characteristic sets
  and functions with prescribed gradient}},
  J. Reine Angew. Math. \textbf{564} (2003), 63--83.

  \bibitem{CCHY}
  J.-H. Cheng, H.-L. Chiu, J.-F. Hwang, and P. Yang,
  \href{https://doi.org/10.1515/crelle-2015-0044}{\emph{Umbilicity and characterization
  of Pansu spheres in the Heisenberg group}},
  J. Reine Angew. Math. \textbf{738} (2018), 203--235.

  \bibitem{CHMY}
  J.-H. Cheng, J.-F. Hwang, A. Malchiodi, and P. Yang,
  \href{https://www.numdam.org/item/ASNSP_2005_5_4_1_129_0/}{\emph{Minimal surfaces in
  pseudohermitian geometry}},
  Ann. Sc. Norm. Super. Pisa Cl. Sci. (5) \textbf{4} (2005), 129--177.

  \bibitem{CHY}
  J.-H.~Cheng, J.-F.~Hwang, and P.~Yang,
  \href{https://doi.org/10.1007/s00208-006-0033-7}{\emph{Existence and uniqueness for
  $p$-area minimizers in the
  Heisenberg group}}, Math. Ann. \textbf{337} (2007), 253--293.

  \bibitem{CowlingKlimaSikora}
  M.~G. Cowling, O. Klima, and A. Sikora,
  \href{https://doi.org/10.1090/S0002-9947-2010-04920-7}{\emph{Spectral
  multipliers for the Kohn sublaplacian on the sphere in $\CC^n$}},
  Trans. Amer. Math. Soc. \textbf{363} (2011), 611--631.

  \bibitem{Folland}
  G.~B. Folland,
  \href{https://doi.org/10.1090/s0002-9947-1972-0309156-x}{\emph{The tangential
  Cauchy--Riemann complex on spheres}},
  Trans. Amer. Math. Soc. \textbf{171} (1972), 83--133.

  \bibitem{HP}
  R.~K. Hladky and S.~D. Pauls,
  \href{https://dergipark.org.tr/en/pub/iejg/article/597626}{\emph{Variation of
  perimeter measure in sub-Riemannian geometry}},
  Int. Electron. J. Geom. \textbf{6} (2013), no.~1, 8--40.
  Earlier version: \href{https://arxiv.org/abs/math/0702237v1}{arXiv:math/0702237v1}
  (2007).

  \bibitem{HurtadoRosales}
  A. Hurtado and C. Rosales,
  \href{https://arxiv.org/abs/1501.04886}{\emph{Existence, characterization and
  stability of Pansu spheres in
  sub-Riemannian 3-space forms}},
  Calc. Var. Partial Differential Equations \textbf{54} (2015), 3183--3227.

  \bibitem{Kato}
  T. Kato,
  \emph{Perturbation Theory for Linear Operators}, 2nd ed.,
  Springer, 1976.

  \bibitem{MontefalconeSpectrum}
  F.~Montefalcone,
  \href{https://arxiv.org/abs/1110.0703v2}{\emph{Hypergeometric solutions of the closed
  eigenvalue problem on Heisenberg Isoperimetric Profiles}},
  \href{https://arxiv.org/abs/1110.0703v2}{arXiv:1110.0703v2 [math.MG]}, 17 November
  2011.

  \bibitem{MontefalconeStability}
  F. Montefalcone,
  \href{https://arxiv.org/abs/1110.0707v2}{\emph{Stability of Heisenberg isoperimetric
  profiles}},
  \href{https://arxiv.org/abs/1110.0707v2}{arXiv:1110.0707v2 [math.MG]}, 17 November
  2011.

  \bibitem{RitoreCalibration}
  M. Ritor\'e,
  \href{https://doi.org/10.1007/s00526-011-0425-8}{\emph{A proof by calibration of an
  isoperimetric inequality in the
  Heisenberg group $\mathbb H^n$}},
  Calc. Var. Partial Differential Equations \textbf{44} (2012), 47--60.

  \bibitem{Szego}
  G. Szeg\H{o},
  \emph{Orthogonal Polynomials}, 4th ed.,
  Amer. Math. Soc. Colloq. Publ., vol.~23,
  Amer. Math. Soc., Providence, RI, 1975.

  \bibitem{VerzellesiVariation}
  S.~Verzellesi,
  \href{https://arxiv.org/abs/2608.03431v1}{\emph{Variation formulas for mean curvature
  functionals}},
  \href{https://arxiv.org/abs/2608.03431v1}{arXiv:2608.03431v1 [math.DG]}, 4 August
  2026.

\end{thebibliography}
\end{document}